\documentclass[11pt,letterpaper]{amsart}

\usepackage[utf8]{inputenc}
\usepackage{amsmath}
\usepackage{amsfonts}
\usepackage{amssymb}
\usepackage{mathrsfs}
\usepackage{graphicx}
\usepackage{bm}

\usepackage[left=3cm,right=3cm, top=3cm,bottom=3cm]{geometry}
\usepackage[colorlinks=true, linkcolor=cyan!20!blue, citecolor=purple, urlcolor=purple]{hyperref}

\usepackage[capitalise]{cleveref}
\usepackage{longtable}

\usepackage{aliascnt}
\usepackage{xparse}

\let\oldnewtheorem\newtheorem
\RenewDocumentCommand{\newtheorem}{s m o m O{}}{%
  \IfBooleanTF{#1}
    {\oldnewtheorem{#2}{#4}}%
    {%
      \IfNoValueTF{#3}%
        {\oldnewtheorem{#2}{#4}[#5]}%
        {%
          \newaliascnt{#2}{#3}%
          \oldnewtheorem{#2}[#2]{#4}%
          \aliascntresetthe{#2}%
        }%
    }%
}

\usepackage{tikz-cd} 
\usepackage[all,cmtip]{xy}	
\xyoption{arrow}
\usepackage{dynkin-diagrams}

\usepackage{xcolor,etoolbox}
\usepackage{enumitem}
\usepackage{adjustbox}

\usepackage[textsize=footnotesize, textwidth=25mm, color=green!40]{todonotes}

\usepackage{array}
\usepackage{multirow}

\theoremstyle{plain}

\newtheorem{theorem}[subsubsection]{Theorem}

\Crefname{theorem}{theorem}{theorems}
\Crefname{theorem}{Theorem}{Theorems}

\newtheorem{lemma}[subsubsection]{Lemma}
\newtheorem{prop}[subsubsection]{Proposition}
\newtheorem{cor}[subsubsection]{Corollary}
\newtheorem{conj}[subsubsection]{Conjecture}

\newtheorem{introthm}{Theorem}[section]

\theoremstyle{definition}

\newtheorem{pp}[subsubsection]{}

\newtheorem{rmk}[subsubsection]{Remark}
\newtheorem{eg}[subsubsection]{Example}

\newcommand{\ZZ}{\mathbb{Z}}

\newcommand{\VV}{\mathbb{V}}
\newcommand{\PP}{\mathbb{P}}

\newcommand{\Gc}{\mathcal{G}}

\newcommand{\Rc}{\mathcal{R}}

\newcommand{\Lc}{\mathcal{L}}

\newcommand{\Hc}{\mathcal{H}}
\newcommand{\Pc}{\mathcal{P}}

\newcommand{\Kc}{\mathcal{K}}
\newcommand{\Oc}{\mathcal{O}}

\newcommand{\Os}{\mathscr{O}}

\newcommand{\Db}{\mathbb{D}}

\newcommand{\Ho}{\mathrm{H}}
\newcommand{\Hom}{\mathrm{Hom}}

\newcommand{\Lie}{\mathrm{Lie}}
\newcommand{\Pic}{\mathrm{Pic}}
\newcommand{\ad}{\mathrm{ad}}
\newcommand{\Spec}{\mathrm{Spec}}
\newcommand{\Res}{\mathrm{Res}}
\newcommand{\Bun}{\mathrm{Bun}}
\newcommand{\Gr}{\mathrm{Gr}}
\newcommand{\Aut}{\mathrm{Aut}}

\newcommand{\SL}{\mathrm{SL}}
\newcommand{\PGL}{\mathrm{PGL}}
\newcommand{\GL}{\mathrm{GL}}

\newcommand{\Gm}{\mathbb{G}_{\mathrm{m}}}
\newcommand{\Iso}{\mathcal{I}\textit{so}}

\newcommand{\g}{\mathfrak{g}}

\newcommand{\Co}{\mathring{X}}

\newcommand{\Cto}{\mathring{C}}
\newcommand{\Ct}{C}

\newcommand{\Gout}{L_{\Co}\Gc}
\newcommand{\LG}{L\Gc}
\newcommand{\LpG}{L^+\Gc}
\newcommand{\GrG}{\Gr_{\Gc}}
\newcommand{\GrGx}{\Gr_{\Gc,x}}
\newcommand{\GrGS}{\Gr_{\Gc,S}}

\newcommand{\Iht}{\widehat{I}_\tau}
\newcommand{\Ihx}{\widehat{I}_x}

\newcommand{\Fl}{\mathrm{Fl}}
\newcommand{\fac}{\mathfrak{F}}

\newcommand{\cc}{\bm{c}}
\newcommand{\cha}{\check{a}}

\newcommand{\hL}{\hat{L}}

\newcommand{\vLambda}{\vec{\Lambda}}
\DeclareMathOperator{\lcm}{lcm}

\newcommand{\Gg}{\mathfrak{G}}
\newcommand{\hGg}{\widehat{\mathfrak{G}}}

\newcommand{\gsd}{\mathrm{gsd}}

\newcommand{\Nf}{\mathfrak{N}}
\newcommand{\Mf}{\mathfrak{M}}

\newcommand\bfem[1]{\textcolor{black}{\emph{\textbf{#1}}}}

\usepackage[justification=centering]{caption}

\author[Chiara Damiolini and Jiuzu Hong]{Chiara Damiolini and Jiuzu Hong \\ {\tiny{with an appendix by}} Shuo Gao}

\address{\textrm{Chiara Damiolini} \newline \indent Department of Mathematics, University of Texas at Austin, Austin, TX}
\email{chiara.damiolini@austin.utexas.edu}

\address{\textrm{Jiuzu Hong} \newline \indent Department of Mathematics, University of North Carolina at Chapel Hill, Chapell Hill, NC}
\email{jiuzu@email.unc.edu}

\address{\textrm{Shuo Gao} \newline \indent Department of Mathematics, Stony Brook University, Stony Brook, NY}
\email{shuo.gao@stonybrook.edu}

\title[Line bundles on the moduli stack of parahoric bundles]{Line bundles on the moduli stack of parahoric bundles}


\usepackage{multibib}
\newcites{main}{References}
\newcites{app}{References for Appendix}

\begin{document}

\begin{abstract} In this paper we investigate line bundles on $\Bun_\Gc$ the moduli stack of parahoric Bruhat--Tits bundles over a smooth projective curve. Translating this problem into one concerning twisted conformal blocks, we are able to establish criteria that detect when line bundles on an appropriate flag variety descend to $\Bun_\Gc$. Along the way we establish a conjecture of Pappas and Rapoport which describes sections of line bundles on $\Bun_\Gc$ using representation-theoretical means. We conclude the paper with examples where our methods allow us to explicitly determine $\Pic(\Bun_\Gc)$.
\end{abstract}

\maketitle

\section{Introduction} 

Parahoric Bruhat--Tits group schemes $\Gc$  and their bundles  \cite{pappas.rapoport:2010:questions,heinloth:2010:uniformization} allow one to put under a unique umbrella various phenomena. Parabolic bundles, equivariant principal bundles, and Prym varieties are some of them \cite{pappas.rapoport:2010:questions,heinloth:2010:uniformization, zelaci:2019, damiolini:2020:conformal, hongkumar:2023, damiolini.hong:2023,Pappas-Rapoport:2022}. This means that understanding the moduli stack $\Bun_\Gc$ of principal $\Gc$-bundles on a smooth and projective curve $X$ implies understanding a wide array of natural construction depending on $X$. Many questions and conjectures concerning $\Bun_\Gc$ and its Picard group were posed by Pappas and Rapoport in \cite{pappas.rapoport:2010:questions}. Since then, many questions have been settled. In this paper we aim to provide an explicit description of line bundles of $\Bun_\Gc$ and their global sections. Specifically, when $\Gc$ is generically simply-connected and simple, and $X$ is a smooth projective curve over an algebraically closed field  $k$ of characteristic zero, Pappas and Rapoport conjecture that sections of line bundles can be described using representation theory and conformal blocks. In this paper we show that indeed this holds:

\begin{introthm}[\Cref{thm:sectionTCB}]
    Global sections of dominant line bundles on $\Bun_\Gc$ are identified with twisted conformal blocks.
\end{introthm}

To state this result precisely and to give an idea of the methods used to prove it, we need to first introduce some notation. Let $S$ be a finite and non empty collection of points of $X$ which contains all the \textit{bad points} $\Rc$ of $\Gc$, namely all $x \in X$ such that the fiber $\Gc|_x$ is non-reductive. Denote by $\Co$ the complement of $S$ in $X$. In view of the uniformization theorem \cite[Theorem 4]{heinloth:2010:uniformization}, the stack $\Bun_\Gc$ can be realized as a quotient of a product of partial affine flag varieties $\GrGS = \prod_{x \in S} \GrGx$ by the group $\Gc(\Co)$ (or more precisely by the ind-group scheme $\Gout$ whose $k$-points are $\Gc(\Co)$). In \cite[Conjecture 3.7]{pappas.rapoport:2010:questions} Pappas and Rapoport conjectured the following statement which we prove here.

\begin{introthm}[\Cref{thm:PRconj}] Let $\Lc$ be a dominant line bundle on $\Bun_\Gc$. Then the projection map $q \colon \GrGS \to \Bun_\Gc$ induces an isomorphism:
    \begin{equation} \label{eq:thmPRintro} \tag{2.2.4}
\Ho^0(\Bun_\Gc, \Lc) \cong \left[  \Ho^0(  \GrGS, q^*\Lc) \right]^{\Ho^0(\Co, \Lie(\Gc))}.
\end{equation} 
\end{introthm}

One of the main tools used to prove our results, is an explicit realization of the restriction of $\Gc$ to $\Co$ in terms of Galois covers. In fact, in \cite[Proposition 6.1.5]{damiolini.hong:2023}, we prove that 
\begin{equation}\label{eq:intro} \tag{$\star$}
    \Gc|_{\Co} \cong \pi_* (G \times \Cto)^\Gamma 
\end{equation} for a finite group $\Gamma$ of diagram automorphisms of $G$ and an étale $\Gamma$-cover $\pi \colon \Cto \to \Co$. (Here $G$ is the simple and simply connected group scheme over $k$ with the same root datum of any geometric fiber of $\Gc$ over $\Co$.) In \Cref{sec:uniqueness} we further comment on the {uniqueness} of the pair $(\Gamma, \pi)$. In particular, we show that the order of $\Gamma$ is an invariant of $\Gc$ which we call the \textit{generic splitting degree} and denote $\gsd(\Gc)$. Note that $\Gc$ is a generically split parahoric Bruhat--Tits group, if and only if  $\Gc|_{\Co}=G \times \Co$ (i.e. $\Gamma$ and $\pi$ are trivial)

Having in hand \eqref{eq:intro} we may apply \cite{hongkumar:2023} to show that $\Gout$ is an \textit{integral} ind-scheme (see \Cref{lem:GoutIntegral}). Not only this suffices to show \Cref{thm:PRconj}, but it also implies that both $\Gout$ and its Lie algebra act on $\Ho^0(\GrGS, q^*\Lc)$ in a \textit{unique} way. Therefore the Picard group of $\Bun_\Gc$ naturally injects, via $q^*$, into the Picard group of $\GrGS$. Said it more geometrically, there is at most one way to descend a line bundle from $\GrGS$ to $\Bun_\Gc$.

The fact that $\Ho^0(\Co, \Lie(\Gc))$ acts in a unique way on $\Ho^0(  \GrGS, q^*\Lc)$ is also the key to show that $\Ho^0(\Bun_\Gc, \Lc)$ can be identified with the space of \textit{twisted conformal blocks}, introduced in \cite{damiolini:2020:conformal} and further generalized and studied in \cite{hongkumar:2023, deshpande.mukhopadhyay:2019, hongkumar:2024}. We first of all note that every line bundle on $\GrGS$ is determined by a collection of weights $\vLambda$. The $\Gamma$ cover $\pi \colon \Cto \to \Co$ used to show \eqref{eq:intro} extends to a (possibly ramified) $\Gamma$ cover $\overline{\pi} \colon \Ct \to X$. We therefore define the parahoric group $\Gc_{\pi} \coloneqq \overline{\pi}_*(G \times C)^\Gamma$ and denote by $\Gg$ its Lie algebra. Given a collection of integrable representations $\Hc(\vLambda)$---where the weights $\vLambda$ determine a unique line bundle $\Lc_{\vLambda}$ on $\GrGS$---one associates the space of \textit{twisted conformal blocks} $\VV^\dagger(\Gg; \vLambda)_{C \to X;S}$ consisting of elements of $\Hc(\vLambda)^*$ which are invariant under the action of $\Gg(\Co)$, in formulas
\begin{equation} \label{eq:CBintro} \tag{CB}
\VV^\dagger(\Gg; \vLambda)_{C \to X;S} \coloneqq [\Hc(\vLambda)^*]^{\Gg(\Co)}.
\end{equation}
We may therefore compare \eqref{eq:thmPRintro} and \eqref{eq:CBintro} and we show that, if $\Lc_{\vLambda}=q^*\Lc$ for a line bundle $\Lc$ on $\Bun_\Gc$, then there is a natural isomorphism
\[  \Ho^0(\Bun_\Gc, \Lc) \cong \VV^\dagger(\Gg;\vLambda)_{C \to X; S}, 
\] proving \Cref{thm:sectionTCB}.

Twisted conformal blocks can be used to detect whether line bundles on $\GrGS$ descend to $\Bun_\Gc$, generalizing a method developed in \cite{sorger:99:moduli}. Using \cite[Proposition 4.1]{zhu:2014:coherence} we show that the image of $\Pic(\Bun_\Gc)$ lies in an explicitly determined subgroup $\Pic^\Delta(\GrGS) \subseteq \Pic(\GrGS)$ (see \Cref{cor:PicDeltainPicBunG}). Using this notation, we can state the following sufficient criterion:

\begin{introthm}[\Cref{thm:criterion}] Let $\Lc_{\vLambda} \in  \Pic^\Delta(\GrGS)$ be a dominant line bundle. If the space of twisted conformal blocks $\VV^\dagger(\Gg; \vLambda)_{C\to X;S}$ is non zero, then $\Lc_{\vLambda}$ descends to a line bundle on $\Bun_\Gc$. \end{introthm}

We can combine this result with the description of $\Pic(\Bun_\Gc)$ given in \cite{heinloth:2010:uniformization}. In fact,  \cite[Theorem 3]{heinloth:2010:uniformization} provides the exact sequence
\begin{equation*}
\xymatrix{
0 \ar[r] & \prod_{x \in \Rc} X^*(\Gc_x)   \ar[r] & \Pic(\Bun_\Gc)  \ar[r]^-\cc & \cc_\Gc\ZZ \ar[r] & 0}\end{equation*}
where $X^*(\Gc|_x)$ is the group of characters of the fiber of $\Gc$ over $x$ and the map $\cc$ is the \textit{central charge}. The obstacle in understanding $\Pic(\Bun_\Gc)$ lies in the computation of the positive integer $\cc_\Gc$ that makes $\cc$ surjective. In view of \Cref{cor:PicDeltainPicBunG} we can show that $\cc_\Gc$ is always a multiple of an explicit integer $\cc_\Delta$ and we conjecture that $\cc_\Gc=\cc_\Delta$. Using \Cref{thm:criterion} we show that this is indeed the case in a number of cases, including the following situation:

\begin{introthm}[Theorems \ref{thm:Iwa} and \ref{appthm:Iwa}] \label{thm:D}
Let $\Gc$ be a parahoric Bruhat--Tits group over $X$. If, for every $x \in \Rc$ the group $\Gc(\Db_x)$ is an Iwahori subgroup of $\Gc(\Db_x^\times)$, then $\cc_\Gc=1$.  
\end{introthm}

In the first version of this paper, we proved Theorem D assuming that the generic splitting degree of $\Gc$ was different from $6$, which recovers \Cref{thm:Iwa}. In this new updated version, we added Appendix~\ref{appendix} by Shuo Gao in which \Cref{appthm:Iwa} covers the case $\gsd(\Gc)=6$. One of the strengths of \Cref{thm:criterion} is the fact that one does not need to explicitly compute spaces of twisted conformal blocks $\VV^\dagger$, but only show that they are non-zero. One of the crucial properties of twisted conformal blocks is that they decompose as a direct sum of \textit{simpler} spaces $\VV^\dagger_i$, and therefore, to apply \Cref{thm:criterion}, it is enough to show non-vanishing for one of such spaces. Following these ideas, in \Cref{subsec:nonvanishing} we show \Cref{thm:Iwa} and, in the appendix, Shuo Gao proves \Cref{appthm:Iwa}. We conclude \Cref{subsec:nonvanishing} with  \Cref{prop:lcmai} and with further examples where $\cc_\Gc$ is computable (\Cref{eg:covering}).

\subsection*{Acknowledgment} C.\,Damiolini warmly thanks the organizers and the participants of the Bootcamp for the 2025 Algebraic Geometry Summer Research Institute. The appendix was made possible by the Bootcamp for the 2025 Algebraic Geometry Summer Research Institute, and S.\,Gao thanks the first author for leading the group session thereof and thanks Davide Gori for some preliminary discussion during the same group session. 

C.\,Damiolini is partially supported by the NSF grant DMS--2401420. J.\,Hong is partially supported by the Simons grant MPS-TSM-00007468. 
S.\,Gao is partially supported by the generosity of Simons Foundation International.

\section{Background and Pappas and Rapoport conjecture} \label{sec:PRconj}

Throughout we let $X$ be a smooth and projective curve over $k$. The field $k$ is assumed to be algebraically closed and with characteristic zero. If $x$ is a point of $X$, we denote by $\Db_x = \Spec(\Oc_x)$ the formal disk around $x$ and $\Db_x^\times= \Spec(\Kc_x)$ the punctured disk around $x$. By choosing a coordinate $t$ around $x$, we have that $\Oc_x \cong k[\![t]\!]$ and $\Kc_x \cong k(\!(t)\!)$. Throughout the paper, by a \bfem{parahoric Bruhat--Tits group scheme} over a curve $X$ we mean an affine and smooth group scheme $\Gc$ over $X$ such that
\begin{enumerate}[label=(\arabic*)]
    \item \label{it:gensimple} $\Gc$ is a generically simply-connected and simple algebraic group;
    \item $\Gc|_x$ is connected for every $x \in X$;
    \item \label{it:parahoric}  $\Gc(\Oc_x)$ is a \emph{parahoric subgroup} of $\Gc(\Kc_x)$ for every $x \in X$, as in \cite{bruhat.tits:1984:II}.
\end{enumerate}
We note that condition \ref{it:parahoric} is automatically verified when $\Gc|_x$ is reductive (and actually simple, by \ref{it:gensimple}) and we refer to \cite{heinloth:2017:stability,damiolini.hong:2023, Pappas-Rapoport:2022} for a more detailed  discussion on the parahoric condition.  Furthermore, by \ref{it:gensimple} there are only finitely many points of $X$ when this happens. We call such points the \bfem{bad points} of $\Gc$ and denote their set by $\Rc$. 
We further say that $x \in X$ is \bfem{split} if $\Gc|_{\Kc_x}$ is a split reductive group over ${\Kc_x}$, otherwise it will be called non-split. In fact, in our setting $\Gc|_{\Kc_x}$ is always quasi-split. The \bfem{splitting degree} of $x \in \Rc$ (relative to $\Gc$) is
\[ \deg_\Gc(x) \coloneqq \min \{ \deg(\Lc_x : \Kc_x) \text{ such that } \Gc_{\Lc_x} \text{ is split } \} \subseteq \{1,2,3\}.\] Therefore $x \in \Rc$ is split if and only if $\deg_\Gc(x)=1$. We extend this notion to every point of $X$ by declaring any point $x \not\in \Rc$ to have splitting degree $0$.
We note that in \cite{heinloth:2010:uniformization}, condition \ref{it:gensimple} is replaced by the weaker hypothesis that $\Gc$ is generically semisimple, while in \cite{balaji.seshadri:2015:moduli} the authors assume that $\Gc$ is generically semisimple and split, so that necessarily all the bad points are split.  Although we are mainly interested in parahoric Bruhat--Tits group schemes which have bad points that are non-split, we will not make this assumption.

\subsection{The stack of \texorpdfstring{$\Gc$}{G}-bundles} We denote the moduli stack of $\Gc$-bundles over $X$ by $\Bun_{\Gc}$.
This is a smooth algebraic stack over $k$ and the uniformization theorem (consequence of \cite[Theorem 4]{heinloth:2010:uniformization}) describes it as an appropriate quotient stack. To understand this characterization, we need to set up some notation. Let $S$ be a non empty collection of points of $X$ and $\Co \coloneqq X \setminus S$. We will later assume that $S$ contains the set $\Rc$ of all the bad points, but for the moment this assumption will not be needed. The ind-group scheme $\Gout$ is defined by the assignment
\[  \Gout(R) \coloneqq \Gc\left(\Co \times_{\Spec(k)} \Spec(R)\right)\] 
for every $k$-algebra $R$. For every $x\in X$, we define the ind-group scheme $\GrGx$, called the \bfem{affine Grassmannian} of $\Gc$ which parametrizes $\Gc$-bundles over $X$ which are trivial outside $x$. When $x \in \Rc$, then $\GrGx$ is a partial affine flag variety. Alternatively, one defines the loop groups $\LG_x$ and $\LpG_x$ by the assignments
\[ \LG_x(R) \coloneqq \Gc(R(\!(t)\!)) \qquad \text{ and } \qquad  \LpG_x(R) \coloneqq \Gc(R[\![t]\!]),
\] where $t$ is a local coordinate of $X$ at $x$. The affine Grassmannian $\GrGx$ can therefore be identified with the quotient $\LG/\LpG$.

From this description, we have that there is a natural map $\Gout \to \LG$ and thus $\Gout$ naturally acts on $\GrGS \coloneqq \prod_{x \in S} \GrGx$ and the Uniformization Theorem \cite[Theorem 4]{heinloth:2010:uniformization} implies that
\begin{equation}
    \label{eq:Heinloth}
    \Bun_\Gc \cong \left[ \Gout \setminus \GrGS \right],
\end{equation}
where the isomorphism is induced by the map
\begin{equation} \label{eq:unif} q_S \colon \GrGS \longrightarrow \Bun_\Gc.
\end{equation}

In \cite[Conjecture 3.7]{pappas.rapoport:2010:questions}, Pappas and Rapoport suggest that the following statement holds.

\begin{conj}\label{conj:PR} Let $\Lc$ be a dominant line bundle on $\Bun_\Gc$. Assume that $\Rc \subseteq S \neq \emptyset$. Then there is a canonical isomorphism
\begin{equation}\label{eq:PRconj}
\Ho^0(\Bun_\Gc, \Lc) \cong \left[  \bigotimes_{x \in S} \Ho^0(\GrGx, q_x^*\Lc) \right]^{\Ho^0(\Co, \Lie(\Gc))},
\end{equation}
where the notation $[X]^Y$ used on the right hand side denotes the elements of $X$ which are annihilated by $Y$. \end{conj}

In \cite[Theorem 12.1]{hongkumar:2023}, Kumar and the second author proved a form of this conjecture for certain parahoric Bruhat--Tits group scheme under some constraint of central charge.   We will show in \Cref{thm:PRconj} that indeed \Cref{conj:PR} holds true in general.

\begin{rmk}  In \cite{pappas.rapoport:2010:questions} it is claimed that the action of the Lie algebra $\Ho^0(\Co, \Lie(\Gc))$ on $\bigotimes_{x \in S} \Ho^0(\GrGx, q_x^*\Lc)$ comes from the fact that the map
\[\Ho^0(\Co, \Lie(\Gc)) \longrightarrow \bigoplus_{x \in S} \Ho^0 \left( \Db_x^\times, \Lie(\Gc)\right)  = \bigoplus_{x \in S}  \Lie(\LG_x)
\] has a unique splitting to a central extension of $\bigoplus_{x \in S}  \Lie(\LG_x)$. This assertion does not look obvious to us, and can be seen as a consequence of \Cref{prop:XGouttrivial}.
\end{rmk}

\subsection{Establishing Pappas and Rapoport conjecture} \label{sec:PRconjProof} We describe in this section the ingredients and main steps used to prove \Cref{thm:PRconj}.  Let $\Gc$ be our parahoric Bruhat--Tits group scheme over $X$ and assume now that $\Rc \subseteq S \neq \emptyset$. As before, we denote by $\Co$ the affine curve $X \setminus S$. Denote by $G$ the simple and simply connected group over $k$ with the same root datum as one (thus all) geometric fiber of $\Gc$ over $\Co$. As customary, its Lie algebra $\Lie(G)$ will be denoted $\g$.

We begin with a crucial result which will be heavily used throughout. We recall that we assume that $\Rc \subseteq S \neq \emptyset$.

\begin{lemma} \label{lem:GoutIntegral} The ind-group scheme $\Gout$ is integral.    
\end{lemma}

\begin{proof} We begin by applying  \cite[Proposition 6.1.5]{damiolini.hong:2023} to the restriction of $\Gc$ to $\Co$ (which is therefore reductive over an affine curve). Namely, that result allows one to show that there exist
\begin{itemize}
    \item a finite group $\Gamma \subset \Aut(G)$ acting on $G$ by diagram automorphisms, and 
    \item an étale $\Gamma$-cover $\pi \colon \Cto \to \Co$
\end{itemize}
such that
\begin{equation} \label{eq:mainGout} \Gc|_{\Co} \cong \pi_* \left(G \times \Cto\right)^\Gamma.
\end{equation} Thus we deduce that 
\[ \Gout(R) = G(\Cto \times_{\Spec(k)} \Spec(R))^\Gamma,
\]and therefore \cite[Theorem 9.5, Corollary 11.5]{hongkumar:2023} guaranteed that the ind-group scheme $\Gout$ is integral.   \end{proof}

\begin{prop}\label{prop:XGouttrivial}
    The only character of the ind-group scheme $\Gout$ is the trivial one. Thus
    \[ q_S^* \colon  \Pic(\Bun_\Gc) \cong 
\Pic_{\Gout}(\GrGS)  \longrightarrow \Pic(\GrGS)
    \] is injective.
\end{prop}

\begin{proof} 
The isomorphism $\Pic(\Bun_\Gc) \cong 
\Pic_{\Gout}(\GrGS)$ immediately follows from \eqref{eq:Heinloth}.  Then, the injectivity of $q_S^*$ follows from the the triviality of any character of $\Gout$.

In view of \Cref{lem:GoutIntegral}, we know that $\Gout$ is integral. We can therefore proceed as in the proof of \cite[Corollary 5.2]{laszlo.sorger:1997}, so that it is left to prove that
\[ [(\g \otimes_k \Os(\Cto))^\Gamma, (\g \otimes_k \Os(\Cto))^\Gamma] = (\g \otimes_k \Os(\Cto))^\Gamma.
\] Since $\Gc$ is generically simple, we can apply \Cref{lem:simpleLie} and therefore conclude. 
\end{proof}

\begin{lemma} \label{lem:simpleLie}
Let $\mathfrak{k}$ be coherent sheaf of Lie algebras over an affine curve $U$ such that every point $x\in U$, the fiber $\mathfrak{k}_x$ is a simple Lie algebra.  Then $[\mathfrak{k}(U),\mathfrak{k}(U)]= \mathfrak{k}(U)$. 
\end{lemma}
\begin{proof}
The Lie bracket $[ \, , \,]$ of $\mathfrak{k}$ induces a morphism of coherent sheaves
\[ L \colon  \wedge^2 \mathfrak{k} \longrightarrow \mathfrak{k},\]
such that $L ( \wedge^2 \mathfrak{k} )=[\mathfrak{k},\mathfrak{k} ]$.
Since every fiber of $\mathfrak{k}$ is simple, $L_x$ is surjective on fibers. Nakayama's Lemma implies therefore that $L$ is surjective and since $U$ is affine we can conclude.
\end{proof}

We now have all the ingredients to prove Pappas and Rapoport conjecture.

\begin{theorem}\label{thm:PRconj} Let $\Lc$ be a line bundle on $\Bun_\Gc$. Then there is a canonical isomorphism
    \begin{equation}\label{eq:PRthm}
\Ho^0(\Bun_\Gc, \Lc) \cong \left[  \bigotimes_{x \in S} \Ho^0(  \GrGx, q_x^*\Lc) \right]^{\Ho^0(\Co, \Lie(\Gc))},
\end{equation} where the action of $\Ho^0(\Co, \Lie(\Gc))$ on $\bigotimes_{x \in S} \Ho^0(\GrGx, q_x^*\Lc)$ is induced by the inclusion $q^*_S$.
\end{theorem}

\begin{proof} Combining \eqref{eq:Heinloth} with the conclusions of \Cref{prop:XGouttrivial}, we immediately deduce that $q_S$ induces the isomorphism
\[\Ho^0(\Bun_\Gc, \Lc) \cong \left[ \Ho^0(\GrGS, q_S^*\Lc) \right]^{\Gout} = \left[  \bigotimes_{x \in S} \Ho^0(\GrGx, q_x^*\Lc) \right]^{\Gout}.
\] Since both $\Gout$ and $\GrGS$ are integral ind-variety (see \Cref{lem:GoutIntegral} and \cite[Corollary 9.8]{hongkumar:2023}), we can apply \cite[Proposition 7.4]{beauville.laszlo:1994:conformal} to conclude that the space of equivariant $\Gout$-sections coincides with the space of sections annihilated by its Lie algebra, thus concluding the argument.
\end{proof}

\section{Picard group of \texorpdfstring{$\Bun_\Gc$}{BunG}} \label{sec:picBun}

With the results from the previous section, we are one step closer to understand the Picard group of $\Bun_\Gc$. Throughout, we use the same notation as in the previous section.

\subsection{More on parahoric groups, affine Grassmannians and central charge} \label{sec:FlagI}  We describe in more explicit terms the affine Grassmannian $\GrGx$ for a fixed point $x$, which thus we drop from the notations and write $\GrG$ (and similarly $\LG$ and $\LpG$) instead.

Since the group $\Gc$ is a parahoric Bruhat--Tits group scheme we obtain that 
\begin{equation}\label{eq:LG} \LG(R) = G(R(\!(z)\!))^\tau\end{equation} for some outer automorphism $\tau$ of $G$ of order $r$ which acts $R$-linearly on $R(\!(z)\!)$ and such that $\tau(z)=\zeta_r z$ with $\zeta_r$ some primitive $r$-th root of one. (Note that we allow $r=1$, in which case $\tau$ is the identity morphism.) Furthermore $\LpG(k) \eqqcolon \Pc$ is a parahoric subgroup of $G(k(\!(z)\!))^\tau$. By abuse of notation we will write
\[ \Fl_{\Pc} = G(k(\!(z)\!))^\tau / \Pc.
\] to denote the ind-scheme $\GrG$.

We now give a more explicit description of the Picard group of $\Fl_{\Pc}$. Every parahoric subgroup $\Pc$ of $G(k(\!(z)\!))^\tau$ is determined by a facet $\fac$ of the Bruhat--Tits building of $G(k(\!(z)\!))^\tau$. The set of the vertices of $\fac$ is canonically in one-to-one correspondence with the set of nonempty subsets of $\Iht$, the set of vertices of the affine Dynkin diagram associated to the twisted loop group $G(k(\!(z)\!))^\tau$. 
When $\g$ is a simple Lie algebra of type $X_N$, then $\Iht$ is the set of vertices of the affine Dynkin diagram of type $X_N^{(r)}$. 

For each facet $\fac$, we denote by $\Pc_\fac$ the associated parahoric subgroup of $G(k(\!(z)\!))^\tau$ and by $Y_\fac$ the associated subset of $\Iht$. With this notation, we define the partial affine flag variety as the ind-scheme
\[ \Fl_\fac \coloneqq \Fl_{\Pc_\fac} = G(k(\!(z)\!))^\tau/\Pc_{\fac}, \]
whose Picard group has been described  by \cite{pappas.rapoport:2008:haines}  and \cite{zhu:2014:coherence} as
\begin{equation} \label{eq:PicGr} \Pic(\Fl_\fac) \cong \bigoplus_{i \in Y_\fac} \ZZ \Lambda_i.
\end{equation} In the above formula, $\Lambda_i$ denotes the $i$-th fundamental weight of the affine Dynkin diagram of type  $X_N^{(r)}$. The decomposition of \eqref{eq:PicGr} allows one to define the \bfem{central charge} of any line bundle on $\Fl_\fac$ as
\begin{equation} \label{eq:ccGr}
    \cc \left(\sum_{i \in Y_\fac} n_i \Lambda_i \right) \coloneqq \sum_{i \in Y_\fac} n_i \cha_i,
\end{equation}
where $\cha_i \in \{1, \dots, 6\}$ are the \textit{dual Kac labels} attached to the vertex $i\in \Iht$ \cite[\S6.1]{kac:1990:infinite}. In particular, we note that $\cha_o=1$ for the special affine note $o\in \Iht$ labeled as in \cite[\S6.1]{kac:1990:infinite}.

\subsection{Some consequences for \texorpdfstring{$\Bun_\Gc$}{BunG}} We now apply the above consideration to $\GrGS$ for a finite set $S$ of points of $X$. For the scope of this section we do not need to further assume that $\Rc \subseteq S$. In the local description given in \eqref{eq:LG}, the outer automorphism $\tau$ and its order will in general depend on the point $x$, so that we will now denote $\Iht$ by $\Ihx$ to stress this dependence. We will also use the notation $\LpG_x(k)=\Pc_x$.

\begin{lemma} \label{lem:PicGRGS} The isomorphism  \eqref{eq:PicGr} induces the identification
\[\Pic(\GrGS) =  \prod_{x \in S} \Pic(\GrGx) = \prod_{x \in S} \Pic(\Fl_{\Pc_x}) = \prod_{x \in S} \big(\bigoplus_{i \in Y_x} \ZZ \Lambda_i \big ), 
\] where $Y_x$ is the subset of $\Ihx$ corresponding to $\Pc_x$.
\end{lemma}

In view of the previous lemma we will denote line bundles on $\GrGS$ by $\Lc_{\vLambda}$, where $\vLambda=(\Lambda_x)_{x \in S}$ are the corresponding weights. The line bundle $\Lc_{\vLambda}$ is called dominant if the weights $\Lambda_x$ are dominant for every $x \in S$.

Recall that every $x \in X$ defines the quotient map $q_x \colon \GrGx \to \Bun_\Gc$ as in \eqref{eq:unif}, so that we have an induced map
\[ q_x^* \colon \Pic(\Bun_\Gc) \longrightarrow \Pic(\GrGx).
\] We define the \bfem{central charge} of a line bundle $\Lc$ over $\Bun_\Gc$ to be 
\begin{equation} \label{eq:ccBunG} \cc(\Lc) \coloneqq \cc(q_x^*\Lc),
\end{equation} for any $x \in X$. The fact that this is well defined, i.e. that the function $\cc$ is constant in $x$ can be found in \cite[Proposition 4.1]{zhu:2014:coherence}.

Assume now that $S$ is a non-empty subset of $X$ containing $\Rc$. Under these assumptions, we can use \Cref{prop:XGouttrivial} and combine it with  with \eqref{eq:ccBunG} to deduce that there are natural obstructions to the surjectivity of $q^*_S$. As $x$ varies in $S$, the central charge maps \eqref{eq:ccGr} naturally induce the map
\[ \cc_S \colon \Pic(\GrGS) = \prod_{x \in S}\Pic( \GrGx) \longrightarrow \bigoplus_{x \in S}\ZZ.
\] Denote by $\Delta \colon \ZZ \to  \bigoplus_{x \in S}\ZZ$ the diagonal inclusion, and define  
\[ \Pic^\Delta(\GrGS) \coloneqq \cc_S^{-1}(\Delta(\ZZ)),
\] that is $\Pic^\Delta(\GrGS)$ is the group of line bundles on $\GrGS$ whose central charge, computed using the various points $x \in S$, is constant. Since  the central charge is constant line bundles on $\Bun_\Gc$, we immediately obtain the following result.

\begin{cor}\label{cor:PicDeltainPicBunG} The image of $q^*_S$ is contained in $\Pic^\Delta(\GrGS)$.    
\end{cor}

We conjecture that indeed $\Pic^\Delta(\GrGS)$ only detects line bundles coming from $\Bun_\Gc$, that is that $q^*_S$ realizes a bijection between $\Pic(\Bun_\Gc)$ and $\Pic^\Delta(\GrGS)$. The central charge morphisms, on $\Pic(\Bun_\Gc)$ and $\Pic^\Delta(\GrGS)$, induce the two exact sequences:
\begin{equation} \label{Heinloth_es} \xymatrix{
0 \ar[r] & \prod_{x \in \Rc} X^*(\Gc_x)   \ar[r] & \Pic(\Bun_\Gc)  \ar[r] & \cc_\Gc\ZZ \ar[r] & 0}\end{equation}
and \begin{equation}\label{eq:PR34} \xymatrix{
0 \ar[r] & \prod_{x \in \Rc} X^*(\Gc_x)  \ar[r] & \Pic^\Delta(\GrGS) \ar[r] & \cc_\Delta\ZZ \ar[r] & 0,
}
\end{equation}
where both $\cc_\Gc$ and $\cc_\Delta$ are positive integers. The exact sequence \eqref{Heinloth_es} is the content of \cite[Theorem 3]{heinloth:2010:uniformization}, while for \eqref{eq:PR34} we refer to \cite[(3.4)]{pappas.rapoport:2010:questions}. The map $q_S^*$, together with the fact that the central charge morphism on $\Bun_\Gc$ is induced from that on $\GrGS$, gives a map of exact sequences
\[ \xymatrix{
0 \ar[r] & \prod_{x \in \Rc} X^*(\Gc_x) \ar@{=}[d]   \ar[r] & \Pic(\Bun_\Gc)  \ar[r] \ar@{^(->}[d]^{q^*_S} & \cc_\Gc\ZZ  \ar@{^(->}[d]^\iota \ar[r] & 0\\
0 \ar[r] & \prod_{x \in \Rc} X^*(\Gc_x)  \ar[r] & \Pic^\Delta(\GrGS) \ar[r] & \cc_\Delta\ZZ \ar[r] & 0,
}
\] where $\iota$ denotes the usual inclusion (that is $\cc_\Gc$ maps to itself).  It is therefore clear that our conjecture would hold true if $\iota$ were an isomorphism, that is if
\begin{equation} \label{eq:conj} \cc_\Gc=\cc_\Delta.\end{equation} Since $\cc_\Delta$ divides $\cc_\Gc$, we can use this fact to give a lower bound on $\cc_\Gc$. Note that $\cc_\Delta$ is explicitly given by
\begin{equation} \label{eq:cDelta}
    \cc_\Delta = \underset{x \in \Rc}{\lcm}  \left( \gcd_{i \in Y_x} \cha_i \right),
\end{equation}
and from the above discussion it follows that:  
\begin{lemma} \label{lem:lowerBound} The positive integer $\cc_\Gc$ is a multiple of $\cc_\Delta=\underset{x \in \Rc}{\lcm}  \left( \underset{i \in Y_x}\gcd \cha_i \right)$.    
\end{lemma}

\begin{rmk} \label{rmk:SvsR}
    Note that to compute $\cc_\Delta$ in \eqref{eq:cDelta} we only consider points in $\Rc$ and not those in $S \setminus \Rc$. The formula will still hold taking $x \in S$ because when $x \in S\setminus \Rc$, then $o \in Y_x$ and therefore $\cha_o=1$. 
\end{rmk}

\begin{eg} In \cite{HongYu:2024:coherence}, Hong and Yu consider the generically simply connected parahoric Bruhat--Tits group scheme $\Gc$ over $\PP^1$ such that $\Gc(\mathcal{O}_0)$ is a parahoric group associated to a facet $\fac$ and $\Gc(\mathcal{O}_\infty)$ is a parahoric group associated to the negative facet $-\fac$. They showed that any line bundle $\mathcal{L}$ on $\Gr_{\Gc, 0}$ descends to a line bundle $\Bun_{\Gc}$. Combining this with \Cref{lem:lowerBound}, it follows that $\cc_\Gc=\cc_\Delta= \underset{i \in Y_0}\gcd \cha_i $. In particular, it implies that $\cc_\Gc=2$,  when \[\Gc=\pi_* (\SL_{2r+1} \times {\PP^1 })^{\tau},\] where $\pi \colon \PP^1 \to \PP^1$ is the $2$-to-$1$ cover ramified at $0$ and $\infty$ and $\tau$ acts on $\SL_{2r+1}$ by the non trivial diagram automorphism.  \end{eg}

\section{Identification with twisted conformal blocks} \label{sec:tcb}

\subsection{Twisted conformal blocks} We recall in this section the construction of twisted conformal blocks and show that they describe global sections of line bundles on $\Bun_\Gc$. We refer to \cite{tuy,damiolini:2020:conformal,hongkumar:2023} for proofs and details.

\pp Let $X$ be a projective curve, $\Gamma$ be a finite group and $\pi \colon C \to X$ a possibly ramified $\Gamma$-cover. Let $\g$ be a simple Lie algebra and assume that $\Gamma$ acts on $\g$. Then we construct the sheaf of Lie algebras \[\Gg \coloneqq \pi_*(\g \otimes \Oc_C)^\Gamma\] over $X$. Note that, for every $x \in X$, the space of sections of $\Gg$ over the punctured disk $\Db_x^\times$ is a Lie algebra of the form $L(\g, \sigma)$ for some $\sigma \in \Aut(\g)$ (which depends on $x$). 
Denote by $\hGg_x$ the central extension of $L(\g, \sigma)$ which is isomorphic to the twisted affine Lie algebra $\hL(\g, \sigma)$.  Explicitly, $\hGg_x$ is the central extension of ($\g \otimes k(\!(z)\!) )^{\sigma}$ whose underlying vector space is \[(\g \otimes k(\!(z)\!) )^{\sigma}  \oplus k,
\] and with Lie bracket given by
\[ [X \otimes f(z), Y \otimes g(z)] = [X,Y] \otimes f(z) g(z) + \dfrac{1}{|\sigma|} \Res(g(z)f'(z)dz) \langle X | Y \rangle,
\] where the form $\langle \; | \; \rangle$ is the unique multiple of the Killing form  $\kappa$ of $\g$ such that $\langle \theta, \theta \rangle =2$ for $\theta$ the highest root of $\g$. Equivalently $2 \check{\g} \langle \; | \; \rangle = \kappa$, for $\check{\g}$ the dual Coxeter number of $\g$. 

\pp Consider now a non-empty subset $S$ of $X$ containing the branch locus of $\pi$ which we denote $\Rc$. For every $x \in S$ denote by $\Hc(\Lambda_x)$ the integrable, irreducible, highest weight representation of $\hGg_x$ associated to the dominant weight $\Lambda_x$ of $\hGg_x$. Set
\[ \Hc(\vLambda) = \bigotimes_{x \in S} \Hc(\Lambda_x) 
\] and assume that $\Hc(\Lambda_x)$ are all representations of the same level. This means that, the central elements of $\bigoplus_{x \in S} \hGg_x$ all act by the same scalar. Therefore $\Hc(\vLambda)$ is a representation of $\hGg_S$, the quotient of $\bigoplus_{x \in S} \hGg_x$ under the identification of the central elements. Denote by $\Gc_S$ the direct sum $\bigoplus_{x \in S} \Gc(\Db_x^\times)$ and set $\Co=X\setminus S$. The residue theorem ensures that the central extension \[ 0 \to k \to \hGg_S \to \Gg_S \to 0\] restricted to the Lie algebra
\[ \Gg(\Co) = (\g \otimes \Oc_C(C\setminus \pi^{-1} S))^\Gamma
\] splits. Therefore we can view $\Gg(\Co)$ as a Lie subalgebra of $\hGg_S$ and so it acts on $\Hc(\vLambda)$.

\pp We define the space of \bfem{twisted conformal blocks} associated with the cover $\pi \colon C \to X$, the Lie algebra $\Gc$, the points $S$ and the representation $\vLambda$ to be the subspace of $\Hc(\vLambda)^*$ which is annihilated by $\Gg(X\setminus S)$, namely
\[ \VV^\dagger(\Gg;\vLambda)_{C \to X;S}  \coloneqq \Hom_{\Gg(\Co)}\left(\Hc(\vLambda), k \right)=\left[\Hc(\vLambda)^*\right]^{\Gg(\Co)}.
\] When $\Gg$ is understood, we will simply denote this space by $\VV^\dagger(\vLambda)_{C \to X;S}$.

As described in \cite{tuy, damiolini:2020:conformal, hongkumar:2023}, these spaces fit together to define a locally free sheaf of finite rank over the moduli stack parametrizing $\Gamma$-covers of curves. Another key properties of conformal blocks, is the fact that they can be identified with the space of global sections of appropriate line bundles on $\Bun_\Gc$ as we describe next.

\subsection{Sections of line bundles as twisted conformal blocks} \label{subsec:gtfCB} We show in this section that, given a parahoric Bruhat--Tits group $\Gc$ over $X$, spaces $\Ho^0(\Bun_\Gc, \Lc)$ can be described via twisted conformal blocks. \Cref{thm:sectionTCB} generalizes \cite{beauville.laszlo:1994:conformal, laszlo.sorger:1997, hongkumar:2023}. 

One of the key input is to use \cite[Proposition 6.1.5]{damiolini.hong:2023}. As already described in the proof of \Cref{lem:GoutIntegral}, this result ensures the existence of a finite group $\Gamma \subseteq \Aut(G)$ acting by diagram automorphisms on $G$ and an étale $\Gamma$-cover $\pi \colon \Cto \to \Co$ such that
\begin{equation} \label{eq:Gout}  \Gc|_{\Co} \cong \pi_* \left(G \times \Cto\right)^\Gamma.
\end{equation} 
Note that since the curve $X$ is proper, we can extend $\Cto$ to a $\gamma$ cover $\bar{\pi} \colon \Ct \to X$, which is ramified exactly over the non-split bad points of $\Gc$. Furthermore, the action of $\Gamma$ on $G$ gives an action of $\Gamma$ on $\g=\Lie(G)$ and therefore we construct the sheaf of Lie algebras 
\begin{equation} \label{eq:Gg}
    \Gg \coloneqq \bar{\pi}_*(\g \otimes \Oc_\Ct)^\Gamma
\end{equation} over $X$. Using this notation, we can state one of the main results:

\begin{theorem} \label{thm:sectionTCB}  Let $\Lc$ be a line bundle on $\Bun_\Gc$ and denote by $\Lc_{\vLambda}$ the associated line bundle on $\GrGS$. Then \eqref{eq:PRthm} induces the identification 
\[  \Ho^0(\Bun_\Gc, \Lc) \cong \VV^\dagger(\Gg;\vLambda)_{C \to X; S}.
\]\end{theorem}

\begin{proof} We first of all recall that the Borel--Weil theorem (or rather its generalization due to Kumar and Mathieu) implies that
$\Hc(\Lambda)^* = \Ho^0(\GrGx, \Lc_\Lambda)$
and therefore
\begin{equation*} \label{eq:BW} \Hc(\vLambda)^* = \Ho^0(\GrGS, \Lc_{\vLambda}). 
    \end{equation*}
In view of \Cref{thm:PRconj}, in order to conclude it is enough to show that the two Lie algebras $\Gg(\Co)$ and $\Ho^0(\Co, \Lie(\Gc))$, which are naturally identified, induce the same action on $\Hc(\vLambda)$. To do so we will need to construct some auxiliary central extensions of Lie algebras (and of groups) which will be used also in the proof of \Cref{thm:criterion}. 

The following lemma is a consequence of \cite[Proposition 10.2 and Theorem 10.3]{hongkumar:2023}.
 
    \begin{lemma}  \label{lem:Faltings} Let $\Lc_{\vLambda} \in  \Pic^\Delta(\GrGS)$. Then the representation $\Hc(\vLambda)$ of $\hGg_S$ can be uniquely integrated to a projective representation of $\LG_S \coloneqq \prod_{x \in S} \LG_x $. That is, there exists a unique map $\rho_{\vLambda} \colon \LG_S \to \PGL(\Hc(\vLambda))$ whose derivative recovers $\Hc(\vLambda)$ as representation of $\hGg_S$.
    \end{lemma}

Therefore, we define the central extension $\LG_S(\vLambda)$ of $\LG_S$ by pulling back the exact sequence
    \[\xymatrix{
    1 \ar[r] & \Gm \ar[r] & \GL(\Hc(\vLambda)) \ar[r] & \PGL(\Hc(\vLambda)) \ar[r] & 1
    }\]
    along the map $\rho_{\vLambda}$, obtaining the exact sequence
    \begin{equation}\label{eq:sesLGS} \xymatrix{
    1 \ar[r] & \Gm \ar[r] & \widehat{\LG}_S(\vLambda) \ar[r] & \LG_S \ar[r] & 1.
    }\end{equation}
Starting with \eqref{eq:sesLGS}, we can restrict it to $\Gout$, and define the central extension
\begin{equation}\label{eq:sesGout}
     \xymatrix{1 \ar[r] & \Gm \ar[r] & \widehat{\Gout}(\vLambda) \ar[r] & \Gout \ar[r] & 1.}
\end{equation}
From the defining properties of  \eqref{eq:sesLGS} and \eqref{eq:sesGout}, it follows that the exact sequences of Lie algebras obtained from \eqref{eq:sesLGS} and \eqref{eq:sesGout} are given by 
\begin{equation*}\xymatrix{
0 \ar[r] & k \ar[r] \ar@{=}[d] & \hGg_S \ar[r] & \bigoplus_{x \in S} \Gg(\Db_x^\times) \ar[r] &  0\\
0 \ar[r] & k \ar[r] & \hGg(\Co) \ar[r] \ar[u] & \Gg(\Co) \ar[r] \ar[u] & 0.
}
\end{equation*} As discussed in the previous section, the residue theorem tells us that the bottom row splits, so that $\hGg(\Co)  = \Gg(\Co) \oplus k$ as Lie algebras. Invoking \Cref{lem:simpleLie}, we also note that this splitting is unique as otherwise we would have non-trivial maps between $\Gg(\Co)$ and $k$. This implies that if we further assume that \eqref{eq:sesGout} splits, inducing therefore an action of $\Gout$ on $\Hc(\vLambda)$, then the action of its derivative on $\Hc(\vLambda)$ coincides with the action of $\Gg(\Co)$. But note that a splitting of \eqref{eq:sesGout} is equivalent to an action of $\Gout$ on $\Lc_{\vLambda}$, which means that the line bundle $\Lc_{\vLambda}$ descends to a line bundle on $\Bun_\Gc$. But this is exactly one of the assumptions of \Cref{thm:sectionTCB} and thus we conclude.
\end{proof}

\begin{rmk} Let us assume now that the group $\Gc$ is of the form $\pi_*(G \times C)^\Gamma$ for some $\Gamma$-cover $\pi \colon C \to X$ (with $G$ simple and simply connected and $\Gamma$ preserving a Borel subgroup of $G$). Under these assumptions, in  \cite{hongkumar:2023} it is introduced the moduli stack ${\rm ParBun}_{\Gc}$ of $\Gc$-bundles with appropriate parabolic structures at $S$. When the level of $\vLambda$ is divisible by the order of $\Gamma$, it is further shown that there is a line bundle $\Lc_{\vLambda}$ on ${\rm ParBun}_{\Gc_\pi}$ and that there is a natural isomorphism $\VV^\dagger(\Gg; \vLambda)_{C \to X;S} \cong \Ho^0({\rm ParBun}_{\Gc}, \Lc_{\vLambda})$ \cite[Theorem 12.1]{hongkumar:2023}. 
Combining this result with \cite[Theorem 12.1]{hongkumar:2023}, one deduces that $\Ho^0({\rm ParBun}_{\Gc_\pi}, \Lc_1) = \Ho^0(\Bun_\Gc, \Lc_2)$ for $\Gc_\pi=\bar{\pi}(G \times C)^\Gamma$ and for appropriate line bundles $\Lc_i$ (see \Cref{subsec:gtfCB} for the notation). 
\end{rmk}

\subsection{Uniqueness of the cover of curve arising from \texorpdfstring{$\Gc$}{G}} \label{sec:uniqueness} We saw that one of the key ingredients used so far is the realization of $\Gc|_{\Co}$ as in \eqref{eq:Gout}. We describe here \textit{how unique} such description is.

Let $D$ denote the group of diagram automorphisms of $G$, namely the group of automorphisms of $G$ preserving $(B, T,\iota)$ where $T \subset B$ is the datum of a maximal torus of $G$, a Borel containing it and $\iota$ is a pinning with respect to $(B,T)$. Let $\Gamma \subseteq D$ and, for any étale $\Gamma$-cover $\pi \colon \Cto \to \Co$, we can simply write 
\[ \pi_*(G \times \Cto)^\Gamma \cong G \times^\Gamma \Cto.
\] In the following lemma, we write $\Iso$ to denote the group of isomorphisms of group schemes over $\Co$.

\begin{lemma}
\label{lem_tor_to_cover}
For every étale $\Gamma$ cover $\Cto \to \Co$, there is an isomorphism of curves
\[ G_{\ad}\, \backslash  \, \Iso\left(G \times \Co, G \times^\Gamma \Cto \right)\cong  D\times ^\Gamma \Cto \] induced by $\Aut(G)/G_\ad = D$ and which is compatible with the left actions of $D$. 
\end{lemma}

\begin{proof} We show that there is a canonical isomorphism étale locally on $\Co$ and that this descends to $\Co$ itself. We note that
\begin{align*} \Iso\left(G \times {\Co},G \times^\Gamma \Cto \right) \times_{\Co} \Cto &= \Iso\left(G \times {\Cto} , (G \times^\Gamma \Cto) \times_{\Co} \Cto \right)\\
& = \Iso \left( G \times{\Cto}, (G \times \Cto \times_{\Co} \Cto) ^\Gamma \right) \\
& = \Iso \left( G \times{\Cto}, (G \times \sqcup_{\Gamma} \Cto)^\Gamma \right)\\
& = \Iso \left( G\times {\Cto}, G \times{\Cto}\right) = \Aut(G) \times {\Cto}.
\end{align*}
Similarly
\begin{align*}
    (D \times^\Gamma \Cto) \times_{\Co} \Cto  = (D \times \Cto \times_{\Co} \Cto)^\Gamma = D \times \Cto.
\end{align*}
Therefore, the isomorphism $G_{\ad}\, \backslash  \,\Aut(G) = D$ gives, étale locally, the wanted isomorphism. The cocycle that gives rise to a global map is the automorphism of \[\Hom_{\Cto \times_{\Co} \Cto}\left( \Aut(G) \times \Cto \times_{\Co} \Cto, D \times \Cto \times_{\Co} \Cto\right)\] which send the map
\[ (\sigma, c, d) \mapsto (\eta_{c,d}(\sigma), c, d) 
\] to the map 
\[ (\sigma, c, d) \mapsto (\eta_{c,d}(\sigma)\gamma^{-1} , c, d), 
\]  where $\gamma$ is the unique element of $\Gamma$ such that $d = \gamma(c)$. 
\end{proof}

Suppose now that there are two identifications 
\begin{equation*} \label{eq:spdeg} \Gc|_{\Co} \cong \pi_* \left(G \times \Cto\right)^\Gamma \qquad  \text{ and } \qquad  \Gc|_{\Co} \cong \pi'_* \left(G \times \Cto' \right)^{\Gamma'},  \end{equation*}
where $\pi \colon  \Cto\to \Co$ (resp. $\pi' \colon  \Cto'\to \Co$ ) is an \'etale $\Gamma$-cover (resp. $\Gamma'$-cover) for a finite subgroup $\Gamma$ (resp. $\Gamma'$) of $D$. Denote by $\Ct$ and $\Ct'$ the closures of $\Cto$ and $\Cto'$ respectively.

\begin{prop}
\label{prop_unique}
There exist an element $\delta\in D$ and an isomorphism $\phi \colon \Ct \to \Ct'$ over $X$ such that \[ \Gamma=\delta \Gamma' \delta^{-1} \qquad \text{ and } \qquad 
 \phi(\gamma\cdot p)= (\delta^{-1} \gamma \delta)\phi(p)\] for every $\gamma\in \Gamma$ and $p\in C$. 
\end{prop}
\begin{proof} Since by assumption
\[ G \times^\Gamma \Cto =\pi_*(G \times \Cto)^\Gamma \cong  \Gc \cong \pi'_*(G \times \Cto')^{\Gamma'}  = G \times^{\Gamma'} \Cto',  
\] we can use \Cref{lem_tor_to_cover} to induce an isomorphism of curves $\psi^\circ  \colon  D\times ^\Gamma \Cto \cong D\times ^{\Gamma' } \Cto'$ which is compatible with the left action of $D$. This naturally extends to a $D$-equivariant isomorphism
\[\psi \colon  D\times ^\Gamma C \cong D\times ^{\Gamma' } C'. \]
The isomorphism $\psi$ sends connected components to connected components and therefore it restricts to an isomorphism $\psi \colon  \Gamma\times^\Gamma C\to \delta\Gamma'\times^{\Gamma'}  C'$ for some element $\delta \in D$. By $D$-equivariance, the stabilizer groups of both components agree. Thus, we must have $\Gamma=\delta \Gamma' \delta^{-1}$.
 
By identifying $C\cong \Gamma\times^\Gamma C$ via $p\mapsto (1,p)$ and similarly $C'\cong \delta\Gamma'\times^{\Gamma'} C'$ via $p'\mapsto (\delta, p')$, the map $\psi$ defines the isomorphism $\phi \colon  C\to C'$ which satisfies $\phi(\gamma\cdot p)= (\delta^{-1} \gamma \delta)\phi(p)$ for every $\gamma\in \Gamma$, and $p\in C$.
\end{proof}

From the above proposition it follows that, to each parahoric Bruhat--Tits group $\Gc$ over $X$, the order of any group $\Gamma$ such that \eqref{eq:Gout} holds is an invariant that only depends on $\Gc$. We call such number the \bfem{generic splitting degree} of $\Gc$ and denote it by $\gsd(\Gc)$. Since $\Gamma$ is a subgroup of the group of diagram automorphisms of $G$, it follows that $\gsd(\Gc)$ takes values $1$, $2$, $3$ or $6$ only. Furthermore, observe that the splitting degree of each point (as defined in \Cref{sec:PRconj}) must necessarily divide $\gsd(\Gc)$. 

We immediately obtain the following corollary. 

\begin{cor} If $G$ is not of type $D_4$, then necessarily $\Gamma =\Gamma'$ and the isomorphism $\phi$ is $\Gamma$-equivariant. If $G$ is of type $D_4$ and $\Gamma$ has order three, then $\Gamma=\Gamma'$.
\end{cor}

We can use \Cref{prop_unique} to compare conformal blocks associated with different Lie algebras and representations. Using the covers $C$ and $C'$ above, we have the two sheaves of Lie algebras
\[
\Gg \coloneqq \pi_*(\g \otimes \Oc_\Ct)^\Gamma,\qquad \text{ and } \qquad   \Gg' \coloneqq \pi_*(\g \otimes \Oc_{\Ct'})^{\Gamma'} \] over $X$. Observe that although in general neither $\Gg$ nor $\Gg$ equal the Lie algebra of $\Gc$, their associated conformal blocks describe sections of line bundles over $\Bun_\Gc$ (see  \Cref{thm:sectionTCB}). Fix an element $\delta \in D$ and an isomorphism $\phi \colon C \to C'$ satisfying  \Cref{prop_unique}. This datum induces a Lie algebra isomorphism between $\Gg$ and $\Gg'$, and also a Lie algebra isomorphism between $\hGg_x=\hL(\g, \tau)$ and $\hGg'_x=\hL(\g,\tau'_x)=\hL(\g, \delta^{-1} \tau_x \delta)$. Therefore to every dominant weight $\Lambda_x$ of $\hGg_x$ there is a unique induced dominant weight of $\hGg'_x$ which we simply denote $\Lambda_x'$. Doing so for every $x \in S$, the pair $(\delta, \phi)$ associates to the collection of dominant weights $\vLambda$ a unique $\vLambda'$. One therefore deduces the following:

\begin{cor} Using the above notation, the pair $(\delta, \phi)$ induces an isomorphism
    \[F_{\delta,\phi} \colon \VV^\dagger(\Gg;\vLambda)_{C \to X; S}
\cong \VV^\dagger(\Gg',\vLambda')_{C' \to X; S}.\]
\end{cor}

\section{Using twisted conformal blocks to control descent}  \label{sec:descent}
In this last section we show how non-vanishing of conformal blocks controls the descent of line bundles from $\GrGS$ to $\Bun_\Gc$. As a consequence, we show that indeed \eqref{eq:conj} holds true in a number of cases.

\subsection{Controlling descent}
We describe here a sufficient condition to ensure that a line bundle on $\GrGS$ descends to $\Bun_\Gc$ or, equivalently, has an $\Gout$-linearization. This is inspired by \cite[Section 3.3]{sorger:99:moduli}, where Sorger uses this result to study line bundles on the stack parametrizing $G$-bundles. As in \Cref{sec:PRconjProof}, we assume that $\Rc \subseteq S \neq \emptyset$. Let $\Gg$ be the Lie algebra over $X$ defined in \eqref{eq:Gg} via any $\Gamma$-cover $C \to X$ such that \eqref{eq:Gout} is satisfied.

\begin{theorem} \label{thm:criterion} Let $\Lc_{\vLambda} \in  \Pic^\Delta(\GrGS)$ be a dominant line bundle. If the space of twisted conformal blocks $\VV^\dagger(\Gg; \vLambda)_{C\to X;S}$ is non zero, then $\Lc_{\vLambda}$ descends to a unique line bundle on $\Bun_\Gc$. \end{theorem}

\begin{proof} We can use \Cref{prop:XGouttrivial} to say that the line bundle $\Lc_{\vLambda}$ descents to $\Bun_\Gc$ if and only if $\Lc_{\vLambda}$ has an $\Gout$-linearization (which is necessarily unique).  As noted in the proof of \Cref{thm:sectionTCB}, this is equivalent to the existence of a splitting (again necessarily unique) of the exact sequence \eqref{eq:sesGout}. Finally, using an argument in \cite[Proposition 3.3]{sorger:99:moduli},  the splitting to \eqref{eq:sesGout} is ensured by the non-vanishing of the space 
\[ \Ho^0\left(\GrGS, \Lc_{\vLambda}\right)^{\Gout} \neq 0,
\]
and the fact that $L_{\mathring{X}}\Gc$ is integral (\Cref{lem:GoutIntegral}). 
We thus conclude using \Cref{thm:sectionTCB}, which identifies that space of sections with the space $\VV^\dagger(\vLambda)_{C \to X;S}$ of twisted conformal blocks.
\end{proof} 

\subsection{Non vanishing criteria} \label{subsec:nonvanishing}

We give in this section some criteria to detect non-vanishing of the space of conformal blocks $\VV^\dagger(\vLambda)_{C \to X;S}$. Combining these results with \Cref{lem:lowerBound} we can provide a more accurate estimate of $\cc_\Gc$ and also show that it coincides with $\cc_\Delta$ whenever all parahoric groups are Iwahori (see \Cref{thm:D}). 

When the group $\Gc$ is already given in terms of Galois covers of curves, one can directly use the results and examples of \cite{deshpande.mukhopadhyay:2019} and \cite{hongkumar:2024} to effectively apply \Cref{thm:criterion}. In fact, \cite[Theorem 1.2]{deshpande.mukhopadhyay:2019} gives the analogue of the Verlinde formula for twisted conformal blocks, extending  \cite{faltings:verlinde}. In theory, one may apply that formula to compute the dimension of twisted conformal blocks and immediately deduce non-vanishing. 

\begin{eg} \label{eg:b8} For instance, from  \cite[Example B.8]{deshpande.mukhopadhyay:2019} one has
\[ \dim \VV^\dagger(A_{2r-1}^{(2)}, \vLambda_o)_{C \to X; \Rc} = 2^g r^{g+n-1},
\] where $C \to X$ is a $2$-to-$1$ cover ramified over $\Rc=\{x_1, \dots, x_{2n}\}$ and where the Lie algebras $\hGg_{x_i}$ are all of type $A_{2r-1}^{(2)}$. Using \Cref{thm:criterion} it therefore follows that $\Lc_{\vLambda_o}$ descends to a line bundle on $\Bun_\Gc$. Note that in this situation $\cc_\Gc=\cc_\Delta=1$. \end{eg} 

In practice, however, applying the Verlinde formula can actually be complicated and, since we only care about whether $\VV^\dagger(\vLambda)_{C \to X;S}$ is zero or not, not always necessary. As shown in \Cref{thm:Iwa}, which together with \Cref{appthm:Iwa} they imply \Cref{thm:D}, we can indeed apply \Cref{thm:criterion} without computing the exact dimension of the appropriate space of conformal blocks. 

\begin{theorem} \label{thm:Iwa} Let $\Gc$ on $X$ be a parahoric Bruhat--Tits group of generic splitting degree different from $6$.  If for every $x \in S$, the group $\Pc_x=\Gc(\Db_x)$ is an Iwahori subgroup of $\Gc(\Db_x^\times)$, then $\cc_\Gc=1=\cc_\Delta$. 
\end{theorem}

\begin{proof} Since all parahoric groups are Iwahori, the special vertex $o$ belongs to $\Ihx$ for every $x \in S$ and therefore the sum of weights
\[ \vLambda_o \coloneqq (\Lambda_o)_{x \in S}
\] defines a line bundle on $\GrGS$ which actually belongs to $\Pic^\Delta(\GrGS)$ and has central charge $1$. We are going to show that this line bundle descends to $\Bun_\Gc$ by showing that the associated space of twisted conformal blocks
\[\VV^\dagger \coloneq \VV^\dagger(\Gg; \vLambda_o)_{C \to X; S}\]
 is non zero.

 \smallskip

\textbf{$\gsd(\Gc)=1$:} Since the generic splitting degree of $\Gc$, it means that $\Gamma$ is trivial and so $\VV^\dagger$ coincides with the space of \textit{untwisted} conformal blocks $ \VV^\dagger(\g \otimes X, \vLambda_o)_{X; S}$. A standard method to check that this is non zero is to degenerate $\VV^\dagger$ (by degenerating the curve $X$) into a vector space of the same dimension but which naturally contains \[ \bigotimes_{x \in S} \VV^\dagger(\g, \Lambda_o)_{\PP^1; 0}.
\] All the components are one dimensional by either a direct computation or \cite[Corollary 4.4]{beauville:96:verlinde}. (One could directly compute the dimension of $\VV^\dagger$ using the Verlinde formula, but since we only need to show that this space is non-zero, we can reduce to a simpler computation.)

 \smallskip

\textbf{$\gsd(\Gc)=2$:} Here $\Ct \to X$ is a $\ZZ/2\ZZ$ cover which is ramified exactly over the non-split bad points $\Rc_{2}$, which necessarily have splitting degree $2$. By the Riemann--Hurwitz formula it follows that $|\Rc_{2}|$ is even, and call $x_1, \dots, x_{2\ell}$ its points. In view of \cite[Theorem 4.7]{hongkumar:2024}, and using propagation of vacua, the space of twisted conformal blocks $\VV^\dagger$ degenerates into a vector space of the same dimension, but which contains (at least one copy of) the space
\[ \bigotimes_{i=1}^\ell \VV^\dagger(\Gg; \Lambda_o, \Lambda_o)_{\PP^1 \overset{ \pi }{\longrightarrow} \PP^1; 0, \infty} \otimes \bigotimes_{x \in S \setminus \Rc_2} \VV^\dagger(\g; \Lambda_o)_{\PP^1;0},\]
where $\pi$ is the unique $2$-to-$1$ cover of $\PP^1$ ramified at $0$ and $\infty$. The space $\VV^\dagger(\Gg; \Lambda_o, \Lambda_o)_{\PP^1 \overset{ \pi }{\longrightarrow} \PP^1; 0, \infty}$ is one dimensional by \cite[Lemma 3.12]{Besson-Hong:2025} or \cite[Lemma 8.7]{deshpande.mukhopadhyay:2019}. The space $\VV^\dagger(\g; \Lambda_o)_{\PP^1;0}$ is also one dimensional as discussed in the previous case.  

 \smallskip

\textbf{$\gsd(\Gc)=3$:} Here $C \to X$ is a $\ZZ/3\ZZ$ cover which is ramified over the non-split points $\Rc_3$ which necessarily have splitting degree $3$.  Denote by $\gamma$ one generator of $\ZZ/3\ZZ$. The set $\Rc_3$ can therefore be partitioned into two sets $\Rc_3^+ \cup \Rc_3^-$ depending on the monodromy of $\gamma$ at those points. The étale cover $\Ct \to X$ induces a representation $\pi_1(X \setminus \Rc) \to \ZZ/3\ZZ$ and therefore we have that $|\Rc_3^+|$ and $|\Rc_3^-|$ agree modulo $3$. We have therefore the following three scenarios:
\begin{enumerate}[label=(\alph*)]
    \item \label{it:a} $|\Rc_3^+| \equiv_{3} 0$: in this case also $|\Rc_3|\equiv_3 0$ and we can partition $\Rc_3$ into triples of points all with the same monodromy;
    \item \label{it:b} $|\Rc_3^+| \equiv_{3} 1$: for every $x^\pm \in \Rc_3^\pm$, we can find a partition of $\Rc_3$ which consists of $\{x^+, x^-\}$ and triples of points all with the same monodromy;
    \item \label{it:c} $|\Rc_3^+| \equiv_{3} 2$: for every $x^\pm, y^\pm \in \Rc_3^\pm$, we can find a partition of $\Rc_3$ which consists of the sets $\{x^+, x^-\}$, $\{y^+, y^-\}$, and triples of points all with the same monodromy.
\end{enumerate}

In case \ref{it:a}, write $|\Rc_3^\pm|=3N^\pm$. It follows from \cite[Lemma 4.3]{hongkumar:2024} and propagation of vacua, that the space of twisted conformal blocks $\VV^\dagger$ degenerates into a vector space of the same dimension, but which contains (at least one copy of) the space
\[ \bigotimes_{i=1}^{N^+ + N^-} \VV^\dagger(\Gg; \Lambda_o, \Lambda_o, \Lambda_o)_{E \overset{q}{\longrightarrow} \PP^1; 0, 1, \infty} \otimes \bigotimes_{x \in S \setminus \Rc_3} \VV^\dagger(\Lambda_o)_{\PP^1; 0},
\] where $q$ is a $3$-to-$1$ cover of $\PP^1$ ramified exactly at $0$, $1$ and $\infty$, all of which have the same monodromy. It follows from \cite[Example B.10]{deshpande.mukhopadhyay:2019} that
\[\dim \left( \VV^\dagger(\Lambda_o, \Lambda_o, \Lambda_o)_{E \overset{q}{\longrightarrow} \PP^1; 0, 1, \infty} \right) = 2,\]
which, together with the already noted fact that $\VV^\dagger(\g; \Lambda_o)_{\PP^1;0} \cong k$, shows that also in this case the space of conformal blocks is non zero.

We conclude considering the case \ref{it:b}, while leaving \ref{it:c} to the reader as it proceeds similarly. Let $|\Rc_3^+|=1 + 3N^+$ and $|\Rc_3^-|=1+3N^-$. Then, again using \cite[Lemma 4.3]{hongkumar:2024} and propagation of vacua, $\VV^\dagger$ degenerates into a vector space of the same dimension and which contains (at least one copy of) the space
\[ \bigotimes_{i=1}^{N^+ + N^-} \VV^\dagger(\Gg; \Lambda_o, \Lambda_o, \Lambda_o)_{E \overset{q}{\longrightarrow} \PP^1; 0, 1, \infty} \otimes \VV^\dagger(\Gg;\Lambda_o, \Lambda_o)_{\PP^1 \overset{\pi}{\longrightarrow} \PP^1; 0, \infty} \otimes  \bigotimes_{x \in S \setminus \Rc_3} \VV^\dagger(\Lambda_o)_{\PP^1,0},
\] where $\pi$ is the only $3$-to-$1$ cover of $\PP^1$ ramified at $0$ and $\infty$ only.  To conclude, it is enough to prove that $\VV^\dagger(\Lambda_o, \Lambda_o)_{\PP^1 \overset{\pi}{\longrightarrow} \PP^1; 0, \infty}$ is non trivial and again \cite[Lemma 3.12]{Besson-Hong:2025} or \cite[Lemma 8.7]{deshpande.mukhopadhyay:2019} shows that it is in fact one dimensional. 
\end{proof}

We summarize here the main ideas behind the proof of \Cref{thm:Iwa}. The main strategy we use to show non-vanishing of $\VV^\dagger =\VV^\dagger(\Gg; \vLambda)_{C \to X; S}$ is to \textit{degenerate} this vector space into another vector space of the same dimension, but that \textit{decomposes} into a direct sum of simpler pieces (of which at least one of them can be shown to not be zero). We refer to \cite{damiolini:2020:conformal,
hongkumar:2023, deshpande.mukhopadhyay:2019, hongkumar:2024} for details.

This method arises from geometry: space of twisted conformal blocks define vector bundle on the space parametrizing $\Gamma$-covers of curves. By fixing appropriate invariants, this moduli space is connected and therefore the rank of the vector bundle of conformal blocks is constant. The degeneration of $\VV^\dagger$ mentioned in the proof of \Cref{thm:Iwa} stems from the degeneration of the $\Gamma$-cover $\Ct \to X$ to another one $\Ct' \to X'$, where $X'$ becomes a (nodal) rational curve. Therefore one has
\[ \dim (\VV^\dagger(\Gg; \vLambda)_{C \to X; S}) = \dim(\VV^\dagger(\Gg'; \vLambda)_{C' \to X'; S})=\dim({\VV^\dagger}'),
\] and so we are left to show that ${\VV^\dagger}'$ is non-zero. To do so, we use the \textit{Factorization Theorem}, which gives a natural decomposition of ${\VV^\dagger}'$ into a direct sum of twisted conformal blocks associated to $\Gamma$-covers of $\PP^1$. 

In the proof of \Cref{thm:Iwa} we saw that, when $\vLambda=\vLambda_o$ (which we could assume since all parahoric groups of $\Gc$ were assumed to be Iwahori), there is at least one component of ${\VV^\dagger}'$ that it is easily computable. However, for general $\vLambda$, this last computation can be quite involved and the degeneration requires further care. This depends on the fact that, although one reduces the computation of conformal blocks to those associated to $\Gamma$-covers $C \to \PP^1$, we cannot in general reduce the complexity of $C$  (see also \cite[Remark B.11]{deshpande.mukhopadhyay:2019}).  

\medskip

Following this idea, in \Cref{prop:lcmai} we are able to give a lower bound on $\cc_\Gc$ when $\gsd(\Gc)=2$ and without imposing that parahoric groups are Iwahori. A similar statement when $\gsd(\Gc)=3$ is possible, but we leave it to the reader.

Let $C \to X$ be a $2$-to-$1$ cover such that \eqref{eq:Gout} holds and denote by $\Rc_2$ the branch locus of $X$, which consists of bad points of $\Gc$ of splitting degree $2$. Let $\tau$ be the generator of the Galois group of $C \to X$, which we identify with a diagram automorphism of $G$. As already discussed in the proof of \Cref{thm:Iwa}, necessarily $|\Rc_2|$ is even. 
Without loss of generality, we also assume that $S_{ 1} \coloneqq S \setminus \Rc_2$ has also even cardinality. We therefore have $|\Rc_2|=2N$ and $|S_{1}| =2 M$. 

For $i \in \{1,\dots, 2N\}$, let $Y_i$ be the subset of $\Iht$ determined by the parahoric subgroup $\Gc(\Db_{x_i})$ of $G(k(\!(t)\!))^\tau$. For every $n\neq m \in \{1, \dots, 2N\}$, define
\[ P_{n,m} \coloneqq \{i \in \Iht \text{ such that }i \in Y_n \cap Y_m\} \subset \Iht.\]
Let $\Nf$ be any partition of $\{1, \dots, 2N\}$ into pairs $\{n_1, n_{2}\}, \dots, \{n_{2N-1}, n_{2N}\}$ and denote by $P_{\Nf}$ the product $P_{n_1, n_2} \times \dots \times P_{n_{2N-1}, n_{2N}}$. 

Similarly, for $i \in \{1,\dots, 2M\}$, let $Y_i$ be the subset of $\hat{I}$ determined by the parahoric subgroup $\Gc(\Db_{x_i})$ of $G(k(\!(z)\!))$. As before, for every $n\neq m \in \{1, \dots, 2M\}$, define
\[ Q_{n,m} \coloneqq \{i \in \hat{I} \text{ such that }i \in Y_n \text{ and } i^* \in Y_m\} \subset \widehat{I},\]
where $i^*=-w_0(i)$ is the image of $i$ under the action of the longest element of the Weyl group of $\g$.  Let $\Mf$ be any partition of $\{1, \dots, 2M\}$ into pairs $\{m_1, m_{2}\}, \dots, \{m_{2M-1}, m_{2M}\}$ and denote by $Q_{\Mf}$ the product $Q_{m_1, m_2} \times \dots \times Q_{m_{2M-1}, m_{2M}}$. 

\begin{prop}\label{prop:lcmai} For every partition $\Nf$ of $\{1, \dots, 2N\}$ into pairs, and for every partition $\Mf$ of $\{1, \dots, 2M\}$ into pairs $\cc_\Gc$ divides 
\[ C_{\Nf, \Mf} \coloneqq \underset{\substack{(i_1, \dots, i_N) \in P_{\Nf}\\ (j_1, \dots, j_M) \in Q_{\Mf}}}{\lcm} \left( \cha_{i_1}, \dots, \cha_{i_N}, \cha_{j_1} \dots, \cha_{j_M}\right).\]
\end{prop}

\begin{proof} 
Let $ \vLambda_{\Nf,\Mf}$ be the collection of affine dominant weights attached to $S$, which assigns the pair of points indexed by $i_n$ the dominant weight $\dfrac{C_{\Nf, \Mf}}{\cha_{i_n}}$ for $1\leq n\leq N$, and assigns the pair of points indexed by $j_m$ the dominant weight $\dfrac{C_{\Nf, \Mf}}{\cha_{j_m}}$ for $1\leq m \leq M$.  It is enough to show that the line bundle on $\GrGS$ corresponding to $ \vLambda_{\Nf,\Mf}$ descends to $\Bun_\Gc$. And again, as before, it is enough to show that the associated space of conformal blocks $\VV^\dagger_{\Nf,\Mf} \coloneqq \VV^\dagger(\Gg, \vLambda_{\Nf,\Mf})_{C \to X; S}$ is non zero. To simplify notation, write  \[\Nf_n \coloneqq \dfrac{C_{\Nf, \Mf}}{\cha_{i_n}} \qquad \text{ and }\qquad \Mf_m \coloneqq \dfrac{C_{\Nf, \Mf}}{\cha_{j_n}}.\]
In view of \cite[Theorem 4.7]{hongkumar:2024}, and using propagation of vacua, the space $\VV^\dagger_{\Nf,\Mf}$ degenerates into a vector space of the same dimension, but containing (at least one copy of) the space
\begin{equation} \label{eq:degCB} \bigotimes_{n=1}^N \VV^\dagger(\Gg; \Nf_n\Lambda_{i_n}, \Nf_n\Lambda_{i_n})_{\PP^1 \overset{\pi}{\longrightarrow} \PP^1;0,\infty} \otimes \bigotimes_{m=1}^{M} \VV^\dagger(\g; \Mf_m \Lambda_{j_m}, \Mf_m \Lambda_{j_m}^*)_{\PP^1;0,\infty},
\end{equation} where $\pi$ is the $2$-to-$1$-cover of $\PP^1$ ramified exactly at $0$ and at $\infty$.
Combining \cite[Lemma 8.7]{deshpande.mukhopadhyay:2019} and \cite[Lemma 4.4]{beauville:96:verlinde}, we again deduce that the space \eqref{eq:degCB} is one dimensional and therefore $\VV^\dagger_{\Nf,\Mf}$ is non zero. \end{proof}

\begin{eg} \label{eg:covering} Let $\Gamma$ be  a subgroup of $D$, the group of diagram automorphism of a simply-connected simple algebraic group $G$. Given a $\Gamma$-cover $\pi \colon  C\to X$, we can define the parahoric Bruhat--Tits group scheme $\Gc \coloneqq \pi_*(G \times C)^\Gamma$ over $X$.  By the exact sequence \eqref{Heinloth_es} determined in \cite{heinloth:2010:uniformization},  we always have $\Pic(\Bun_{\Gc})= \cc_\Gc \ZZ$.  
Let $\Rc$ be the branch locus of $\pi$ in $X$ and observe that $\Rc$ is exactly the set of bad points of $\Gc$. Note that for every $x \in \Rc$, the group $\Gc(\Db_x)$ is a special parahoric subgroup of $\Gc(\Db_x^\times)$.

Suppose that $|\Gamma|$ is $2$. Applying \Cref{prop:lcmai}---and enlarging $S$ if necessary---we can deduce that $c_\Gc=1$ whenever $G \neq A_{2n}$ or if $\Rc$ is empty; otherwise,  $c_{\Gc}=2$.

Suppose that $|\Gamma|$ is $3$, in which case $G$ is of type $D_4$. Using the same arguments as in the proof of \Cref{thm:Iwa}, we  have $c_{\Gc}=1$. \end{eg}

\newpage

\appendix

\section{Generic splitting degree 6}
\label{appendix}

\centerline{\sl by Shuo Gao}

\medskip


\newcommand{\bbar}[1]{\bar{\bar{#1}}}

\newcommand{\ra}{\rightarrow}
\newcommand{\la}{\leftarrow}
\newcommand{\La}{\Leftarrow}
\newcommand{\Ra}{\Rightarrow}
\newcommand{\lra}{\leftrightarrow}
\newcommand{\Lra}{\Leftrightarrow}

\newcommand{\mfA}{\mathfrak{A}}
\newcommand{\mfS}{\mathfrak{S}}
\newcommand{\VVd}{\mathbb{V}^\dagger}

\newcommand{\rank}{\mathrm{rk}}

\newcommand{\ibid}{\textit{ibid.}}
\newcommand{\Iwanum}{\Cref{thm:Iwa}} 
\newcommand{\sectDes}{\Cref{sec:descent}}

In this appendix we show that indeed the conclusion of {\Iwanum} holds also for generic splitting degree 6, concluding the proof of \Cref{thm:D} from the introduction.

\begin{theorem}
\label{appthm:Iwa}
    Let $\Gc$ on $X$ be a parahoric Bruhat--Tits group of generic splitting degree $6$. If for every $x \in S$ the group $\Pc_x=\Gc(\Db_x)$ is an Iwahori subgroup of $\Gc(\Db_x^\times)$, then $\cc_\Gc = 1 = \cc_\Delta$.
\end{theorem}

Recall the method summarized after the proof of theorem \Iwanum.
In the cyclic cover case, we can use \cite[\S 4]{hongkumar:2024} and factorization of coinvariants to reduce our calculation to a few simple base cases. Then, invoking \Cref{thm:criterion}, we are left to compute that the dimensions of the spaces of conformal blocks in the base cases are all non-zero. 
What remains to show is the $\mfS_3$ cover case of generic splitting degree 6. It is unknown to the authors whether the Hurwitz stack (of a fixed ramification data) for $\mfS_3$ is connected or not, so alas we could not degenerate this cover directly. We could fix this by \cite[Remark 1.4]{deshpande.mukhopadhyay:2019}, which tells us that the dimension of the conformal blocks associated to points of the Hurwitz stack is equal across all the connected components.
With this in mind, our analogous lemmas to \cite{hongkumar:2024} will be \textit{a priori} weaker.

\subsection{Reduction Lemmas}
For each $n \in\ZZ_{>1}$, let $C_n :=<\zeta_n|\zeta^n_n>\cong \ZZ/n\ZZ$ denote the cyclic group of order $n$ with generator $\zeta_n$. We can develop (weaker) analogues of Lemma 4.2 and 4.3 of \cite{hongkumar:2024} for any semidirect product of cyclic groups $\Gamma=C_n\rtimes C_m$ with $n,m$ coprime, where we denote the action of $C_m$ on $C_n$ by ``$\cdot$'' (i.e. $\zeta_m\cdot \zeta_n:=\zeta_m\zeta_n\zeta_m^{-1}$). Much of the proof is a modification of \cite{hongkumar:2024}.

\begin{lemma}
\label{lem:analogue4.2}
    Let $(\Sigma, \vec{p})$ be an irreducible $s$-pointed smooth $\Gamma$-curve with ramification data $\vec{\gamma}$ such that $\Gamma\cdot \vec{p}$ contains all the ramified points in $\Sigma$. Assume that the quotient $\bbar{\Sigma}=\Sigma/\Gamma$ has genus $\bar{\bar{g}} \geq 1$ (in particular, $(\Sigma, \vec{p})$ is a stable $s$-pointed $\Gamma$-curve).
    Then, we can construct a stable $s$-pointed $\Gamma$-curve $(\Sigma'', \vec{p}'')$ (in particular, $\Sigma''$ is connected) such that the nodal points of $\Sigma''$ form a single $\Gamma$-orbit $\Gamma\cdot y$, the action of $\Gamma$ on $\Gamma\cdot y$ is free and the quotient  $(\Sigma''/\Gamma, \vec{\bbar{p}}'')$ is a degeneration of the stable $s$-pointed curve $(\Sigma/\Gamma, \vec{\bbar{p}})$.
    Moreover, $\Gamma\cdot \vec{p}''$ contains all the ramified points of $\Sigma''$.
    
    If $\bar{\bar{g}}\geq 2$ or if $\bar{\bar{g}}= 1$ and $\{\gamma_1,...,\gamma_s\}$ generate $\Gamma$, then $\Sigma''$ can be taken to be irreducible.
    In any case, $\Sigma''/\Gamma$ is irreducible and hence we can take $\vec{p}''$ to lie in an irreducible component of $\Sigma''$.
\end{lemma}
\begin{proof}
    By assumption $\Sigma\ra \bbar{\Sigma}$ is a $\Gamma$-cover, so we know $C_n \trianglelefteq \Gamma$ acts on $\Sigma$, and we could consider the (categorical/GIT) quotient $\bar{\Sigma} = \Sigma/C_n$.
    As we also know that $\Sigma\ra \bbar{\Sigma}$ is $\Gamma$-invariant, we know $\Sigma\ra \bbar{\Sigma}$ factors through $\Sigma\ra \bar{\Sigma}$. The $\Gamma$-action on $\Sigma$ then descends to a $C_m=\Gamma/C_n$-action on $\bar{\Sigma}$.
    It is clear that both $\Sigma\ra\bar{\Sigma}$ and $\bar{\Sigma}\ra\bbar{\Sigma}$ satisfy the definition of \textit{admissible} in \cite[Appendix A]{deshpande.mukhopadhyay:2019} because there is no node on $\Sigma$ and $\bbar{\Sigma}$.
    We denote the pointed stable curves by $(\bar{\Sigma},\vec{\bar{p}})$ and $(\bbar{\Sigma},\vec{\bbar{p}})$ with genus $\bar{g},\bbar{g}$ respectively and $s,r$ marked points respectively.
    By construction, this intermediate cover separates ramification data within $C_n$ and ramification data within $\Gamma/C_n\cong C_m$.
    We denote the ramification data for $\Sigma\ra \bar{\Sigma}$ and $\bar{\Sigma}\ra\bbar{\Sigma}$ by $\vec{\bar{\gamma}}$ and $\vec{\bbar{\gamma}}$ respectively.
    
    Let $(C',\vec{\bar{\bar{p}}}')$ be a stable degeneration of $(\bbar{\Sigma}, \vec{\bar{\bar{p}}})$ with $C'$ irreducible and with one single node $\bbar{x}$ (which is possible since $\bbar{g}\geq 1$).
    Let $\tilde{C}'$ be the normalization of $C'$ with $\bar{\bar{x}}^+,\bar{\bar{x}}^-$ over $\bar{\bar{x}}$.
    Then $\tilde{C}'$ is smooth and irreducible of genus $\bar{\bar{g}}-1$.
    Let $\bbar{U}$ be the complement
    \[C'\setminus[\bbar{x},\bbar{p}'_1,\cdots, \bbar{p}'_s] = \tilde{C}'\setminus [\bbar{x}^+,\bbar{x}^-,\bbar{p}'_1,\cdots, \bbar{p}'_s].\]
    We first give an overview of the proof for cyclic cover from \cite{hongkumar:2024}, because both the original Lemma 4.2 and its proof is used in this proof.
    In \cite{hongkumar:2024}, the proof to construct a cyclic covering for the degeneration nodal curve is in the following three steps.
    \begin{enumerate}
        \item We first construct a $\Gamma$-bundle on the complement $U$ of the nodal point and all the marked points. To construct such a $\Gamma$-bundle is same as to construct a group homomorphism $\pi_1(U)\ra \Gamma$.
        \item  We then take the unique smooth projective closure of the ambient variety of the bundle to get a cyclic cover of the normalization of the degeneration, ramified only at the marked points.
        \item We finally identify pairs of pre-images of the nodal point according to $\Gamma$ to get a covering which is etale over the nodal point and is a connected stable cover when the additional requirement is met.
    \end{enumerate}
    
    In our proof here, we first carry out all the above steps to the covering $\bar{\Sigma}\ra \bbar{\Sigma}$, obtaining a $C_m$-covering $\bar{\Sigma}'\ra C'$.
    We then apply a $C_m$-equivariant version of step (1) to an appropriate open $\bar{U}$ of $\bar{\Sigma}'$; 
    by $C_m$-equivariant version of step (1) we mean that we construct a $C_m$-equivariant $C_n$-bundle which is the same as to construct a $C_m$-equivariant group homomorphism $\pi_1(\bar{U})\ra C_n$, with the $C_m$-acion on $C_n$ given by conjugation in $\Gamma$. 
    We then use step (2) and (3) again to get a $\Gamma$-covering of the degeneration stable curve.
    This $C_m$-equivariance provides an accurate $\Gamma$-bundle for the composition of two cyclic coverings and represents the main modification to \cite{hongkumar:2024}.

    As in the proof of \cite[Lemma 4.2]{hongkumar:2024}, we can define a group homomorphism $f' \colon \pi_1(\bbar{U})\ra C_m$
    which gives a $C_m$-bundle $\tilde{\bbar{U}}\ra \bbar{U}$ with $\tilde{\bbar{U}}$ a smooth curve.
    We then get a smooth $s$-pointed $C_m$-cover $\pi: \bar{\Sigma}_{f'} \ra \tilde{C}'$ with marked points $\vec{\bbar{p}}'$, such that the ramification data at $\vec{\bbar{p}}'$ is $\vec{\bbar{\gamma}} = (\bbar{\gamma}_1,\cdots, \bbar{\gamma}_s)$ and the ramification data above $\bbar{x}^\pm$ is trivial; note that $\bar{\Sigma}_{f'}$ is not necessarily connected.
    Let $y^\pm$ be points above $\bbar{x}^\pm$ respectively, chosen so that $y^+$ and $y^-$ are in the same component of the curve $\bar{\Sigma}_{f'}$. Thus, $\pi^{-1}(\bbar{x}^+) = \{\zeta_m^i\cdot y^+ | 0 \leq  i \leq m - 1\}$ and $\pi^{-1}(\bbar{x}^-) = \{\zeta_m^i\cdot y^- | 0 \leq  i \leq m - 1\}$ are free $C_m$-orbits.
    By identifying $\gamma^i\cdot y^+$ and $\gamma^{i+1}\cdot y^-$, for each $0\leq i\leq m-1$, we get a stable (in particular connected) curve $\bar{\Sigma}'$ from $\bar{\Sigma}_{f'}$ whose quotient by $C_m$ is exactly $C'$.
    
    Furthermore, we want to construct a $C_m$-equivariant $C_n$-cover over $\bar{\Sigma}'$, where the $C_m$ action on $C_n$ is given by the semi-direct product structure (or the conjugacy action of $C_m$ on $C_n$).
    We start by constructing a $C_m$-equivariant $C_n$-bundle over the complement $\bar{U}$ given by
    \begin{align*}
 \bar{U} \coloneqq \bar{\Sigma}_{f'}&\setminus[(\zeta_m^i\cdot \bar{y}^+)_{i=0}^{m-1}, (\zeta_m^i\cdot \bar{y}^-)_{i=0}^{m-1}, C_m\cdot\bar{p}'_1,\cdots, C_m\cdot\bar{p}'_s]&\\
    =\bar{\Sigma}'&\setminus [(\overline{\zeta_m^i\cdot y^+})_{i=0}^{m-1}, C_m\cdot \bar{p}_1',\cdots ,C_m\cdot \bar{p}'_s]&,
    \end{align*}
    where $\overline{\zeta_m^i\cdot y^+}$ means the image of $\zeta_m^i\cdot y^+$ in $\bar{\Sigma}'$ after the identification.
    Note that $C_m\cdot\vec{\bar{p}}'$ now contains all the possible ramification points.    
    Similar to \cite{hongkumar:2024}, $\pi_1(\bar{U})$ has the presentation 
    \begin{align*}
     \pi_1(\bar{U}) = <(\bar{\alpha}_i),(\bar{\beta}_i),(\delta^+_j),(\delta^-_j),&C_m\cdot\bar{\eta}_1,\cdots,C_m\cdot\bar{\eta}_s|\\(\prod_i [\alpha_i,\beta_i])
     &(\prod_j\delta^+_j)(\prod_j \delta^-_j)
     (\prod_j(\prod_{i=1}^s(\zeta_m^j\cdot\bar{\eta}_i)) =1>,
     \end{align*}
    where $\zeta_m^j\cdot \bar{\eta}_i$ represents the loop around $\zeta_m^j\cdot\bar{p}''_i$, $\delta^\pm_j$ represent the loops around $\zeta_m^j\cdot y^\pm$ and $\bar{\alpha}_i,\bar{\beta}_i$ represent loops around each handle of $\Sigma_{f'}$.
    Note that the action of $C_m$ on $\bar{U}$ induces an action of $C_m$ on $\pi_1(\bar{U})$, and in particular this action will permute the generators $\{\bar{\alpha}_i,\bar{\beta}_i\}$.
    By swapping the labels $\alpha_i$ with $\beta_i$ if necessary, we can assume $C_m$ sends $\alpha$'s to $\alpha$'s, and $\beta$'s to $\beta$'s, or that $C_m$ permutes the pairs $\{(\bar{\alpha}_i,\bar{\beta}_i)\}$.
    We now construct a $C_m$-equivariant group homomorphism $f'':\pi_1(\bar{U})\ra C_n$ such that
    \begin{itemize}
        \item $f''(\zeta_m^j\cdot\bar{\eta}_i) = \zeta_m^j\cdot\bar{\gamma}_i$ for any $1\leq i\leq s$ and for any $0\leq j< m$,
        \item $f''(\delta^\pm_i) = 1$, and
        \item within each $C_m$-orbit $\{(\zeta_m^i \bar{\alpha}_j,\zeta_m^i \bar{\beta}_j)\}$, we let $f''(\bar{\alpha}_j)=\zeta_n,f''(\bar{\beta}_j) = \zeta_n^{-1}$
        and define the value of the rest via the $C_m$-action, i.e. 
        \[f''(\zeta_m^i\bar{\alpha}_j) = \zeta_m^i\cdot f''(\bar{\alpha}_j)=\zeta_m^i\cdot \zeta_n,
        \quad f''(\zeta_m^i\bar{\beta}_k) = \zeta_m^i\cdot \zeta_n^{-1}.\]
    \end{itemize}
    Recall from the proof of \cite[lemma 4.2]{hongkumar:2024} that multiplying all the cyclic ramification data produces the identity.
    We thus have $(\prod_j(\prod_{i=1}^s(\zeta_m^j\cdot\bar{\gamma}_i))= 1$ for cyclic cover $\Sigma\ra \bar{\Sigma}$, so $f''$ is indeed a group homomorphism. 
    We can verify that $C_m$ acts equivariantly on the generators;
    by construction, we have
    \[f''(\zeta_m\cdot \bar{\eta}_i) = \zeta_m \cdot \bar{\gamma}_i, \quad f''(\zeta_m\cdot \alpha_i)= \zeta_m\cdot f''(\alpha_i) \text{ and } f''(\zeta_m\cdot \beta_i)= \zeta_m\cdot f''(\beta_i);\]
    $C_m$ also permutes $\delta^\pm_i$, but $\zeta_m\cdot f''(\delta_i^\pm)=1=f''(\zeta_m\cdot \delta^\pm_i)$.
    This gives rise to a $C_m$-equivariant $C_n$-bundle $\tilde{\bar{U}}\ra \bar{U}$ with $\tilde{\bar{U}}$ again a smooth (but not necessarily connected) curve.
    By taking the unique smooth projective closure $\Sigma_{f''}\supset \tilde{\bar{U}}$, we get a smooth $s$-pointed $C_n$-cover $\pi:\Sigma_{f''}\ra \bar{\Sigma}_{f'}$, and the ramification data above all the $\zeta_m^i\cdot \bar{y}^\pm$'s are trivial.
    Let $z^\pm$ be points above $\bar{y}^\pm$ respectively, chosen so that $z^+$ and $z^-$ are in the same component of the curve $\Sigma_{f''}$; this is possible because $\bar{y}^\pm$ are in the same connected component of $\bar{\Sigma}_{f'}$.
    Note that we have an action of $C_m$ on $\Sigma_{f''}$. We thus identify points $\zeta_n^i\cdot (\zeta_m^i\cdot z^+)$ and $\zeta_n^{i+1}\cdot(\zeta_n^{i+1}\cdot z^-)$ on $\Sigma_{f''}$ to form $\Sigma''$;
    these nodes after such an identification will form the preimage of nodes of $\bar{\Sigma}'$
    because $\zeta_m^i\cdot y^+$ and $\zeta_m^{i+1}\cdot y^-$ are identified to form the node of $\bar{\Sigma}'$.
    Thus $\Sigma''\ra \bar{\Sigma}'$ is also a $C_m$-equivariant $C_n$-bundle.
    
    Because $n$ and $m$ are coprime to each other, sequence $(\zeta_n^i\zeta_m^i)_{i=1}^{mn}$ enumerates every element of $\Gamma$ even though $\zeta_n\zeta_m$ does not generate $\Gamma$ unless $\Gamma = C_n\times C_m$. This ensures that the curve $\Sigma''$ is connected.     
    We therefore get a stable (in particular connected) $s$-pointed $C_n$-curve $\Sigma''$ whose quotient by $C_n$ is exactly $\bar{\Sigma}'$, and whose further quotient by $C_m$ is $C'$.
    The fact that $\Sigma''\ra \bar{\Sigma}'$ is $C_m$ equivariant ensures that the quotient of $\Sigma''$ by $\Gamma$ is exactly $C'$.
    \end{proof}

\begin{rmk}
    This proof utilizes the ideas of the proof of \cite[Lemma 4.2]{hongkumar:2024} and could potentially work for any solvable group $G$ such that the consecutive quotient of the derived sequence $G^{(n-1)}/G^{(n)}$ is cyclic (later referred as a cyclic solvable group).
    This is a constructive proof so we need the consecutive quotient to be cyclic in order to construct the bundle over the open complement of the nodal and marked points.
    We choose our setting here because we only need the result for $\mfS_3$.
    
    We only use the fact that the order of the two cyclic groups are coprime in step (3) to ensure the connectedness of our cover over the nodal curve when the requirement is met. To ensure connectedness for a cyclic solvable group $G$ covering, we need to make similar assumptions that the orders of the cyclic consecutive quotients need to be pairwise coprime. 
\end{rmk}

\begin{lemma}
\label{lem:analogue4.3}
    Let $(\Sigma, \vec{p})$ be a stable $s$-pointed (irreducible) smooth $\Gamma$-cover of $(\PP^1, \vec{p})$ (in particular, $s \geq 3$) such that $\Gamma\cdot\vec{p}$ contains all the ramified points in $\Sigma$ and has ramification data $ \vec{\gamma}= (\gamma_1,..., \gamma_s)$.
    Suppose that $\gamma_1\gamma_2\cdots \gamma_t = 1$ for all pairs $i,j\in \{1,\cdots,t\}$ with both $t, s-t \geq 2$. 
    Then, we can construct a stable $s$-pointed $\Gamma$-curve $(\Sigma', \vec{p}')$ whose quotient is a union of two projective lines intersecting at a point $x$, such that above one projective line the ramification data is $(\gamma_1,...,\gamma_t)$, and above another projective line the ramification data is $(\gamma_{t+1},...,\gamma_s)$ (which makes this quotient a degeneration of $(\PP^1,\vec{p})$).
    Moreover, the fiber over $x$ is a free $\Gamma$-orbit consisting of all the  nodal points of $\Sigma'$.
    Further, $\Gamma\cdot\vec{p}'$ contains all the ramified points of $\Sigma'$.
    
    If $\{\gamma_1,...,\gamma_s\}$ generate $C_n$ or a conjugate of $C_m$, then the curve over the first projective line can be taken to be irreducible.
\end{lemma}
\begin{proof}
    Examining the proof of \cite[Lemma 4.3]{hongkumar:2024}, no information of $\Gamma$ is being used.
    Note that the ramification data $\gamma_i$'s and the expression $\gamma_1\cdots\gamma_t$ depend on the ordering of the marked points.
    The original construction works for any finite group in $\Gamma$.
\end{proof}

\subsection{Reduction of ramification data}
    By propagation of vacua \cite{hongkumar:2023}, we can drop all the marked points with trivial ramification data.
    Using \Cref{lem:analogue4.2} and factorization of coinvariants, we are reduced to compute the dimension of conformal blocks only when the base curve to genus $0$.
    We can apply \Cref{lem:analogue4.3} to the base $\PP^1$ for further reduction, and are reduced to compute the dimension of (twisted) conformal block in the following four base cases: 
    \begin{enumerate}
        \item $\PP^1$ with only 2-ramification points and at most two of them,
        \item $\PP^1$ with only 3-ramification points and at most three of them,
        \item $\PP^1$ with three marked points with ramification data (up to conjugacy) \[\vec{\gamma}=((12),(23),(132)),\] with connected cover of genus $0$ due to Riemann-Hurwitz, and
        \item $\PP^1$ with four marked points with ramification data (up to conjugacy) \[\vec{\gamma}=((12),(23),(123),(123)),\] with connected cover of genus $2$ due to Riemann--Hurwitz.
    \end{enumerate}
    These four cases are basic in the following sense:
    \begin{lemma}
        Given any $(\PP^1,\vec{p})$ with cover $(\Sigma,\vec{q})$ and ramification data $\vec{\gamma}$ (not all trivial), we can find a finite set of curves $\{C_i
    :=(\PP^1,\vec{p_i})\}_{i=1}^m$, with the ramification data $\vec{\gamma}_i$ for $C_i$ matching one of the four base cases listed above,
    such that $\coprod (C_i,\vec{p}_i)$ is the normalization of a degeneration of $\PP^1,\vec{p})$ and $\cup_{i=1}^m \vec{\gamma_i} = \vec{\gamma}$.

    Moreover, there can be at most one of the $C_i$ not belonging to case (1) and (2).
    \end{lemma}

    \begin{proof}
    Suppose we have $(\PP^1;\vec{p})$ with ramification data $\vec{\gamma}$.
    Note that $\pi_1(\PP^1 \setminus \vec{p})$ has presentation
    \[<\eta_1,...,\eta_s|\eta_1\cdots\eta_s =1>,\]
    where each $\eta_i$ is the small loop around $p_i$,
    so we have $\gamma_1\cdots\gamma_s = 1$.
    We start by ordering the marked point so that the ramification data $ \{\gamma_1,\cdots,\gamma_t\}$ contains all the 2-ramification data or transpositions, and $\{\gamma_{t+1},\cdots,\gamma_s\}$ contains all the 3-ramification data or 3-cycles.

    By computing the sign of $\gamma_1\cdots\gamma_s$, we get $t$ must be even. 
    When $t> 3$, by the Pigeonhole principle there must be two identical elements in $\gamma_i$'s.
    We can switch them both to the first two elements via the equality $\sigma \tau = \tau (\tau^{-1}\sigma\tau)$.
    We can then apply \Cref{lem:analogue4.3} to separate the first two marked points onto a single $\PP^1$.
    Thus, we have our first base case: $\mfS_3$-cover of $\PP^1$ with only 2-ramified points and at most two of them.

    Now we have reduced to the case $t=2$ by removing all pairs of transpositions.
    In this case, we have $\gamma_1\gamma_2\in \mfA_3\cong \ZZ/3\ZZ$.
    This means $(\gamma_1\gamma_2)\gamma_3\cdots \gamma_s$ is equivalent to the product of $s-1$ elements in $\mfA_3$.
    Since $\mfA_3\cong \ZZ/3\ZZ$ has order 3,
    we can split $\gamma_s$'s onto many $\PP^1$'s such that each $\PP^1$ has at most 3 marked point (formally identify $p_1,p_2$ as one marked point) with ramification data.
    This gives us our second base case: $\mfS_3$-cover of $\PP^1$ with only 3-ramified points and at most three of them.

    We are left with three cases: $\gamma_1\gamma_2=1,\gamma_1\gamma_2\gamma_3 = 1$ or $\gamma_1\gamma_2\gamma_3\gamma_4=1$ with $\gamma_1,\gamma_2$ 2-ramification data and $\gamma_3,\gamma_4$ 3-ramification data.
    When $\gamma_1\gamma_2 = 1$, this just ``decomposes'' into case (1) and case (2).
    Otherwise, up to conjugacy, we have
    \[\gamma_1\gamma_2\gamma_3 = 1 \quad \text{ which implies } \quad  \gamma_1 = (12),\gamma_2 = (23),\gamma_3 = (132),\]
    and that
    \[\gamma_1\gamma_2\gamma_3\gamma_4 = 1\quad \text{ which implies } \quad \gamma_1=(12),\gamma_2 = (23),\gamma_3 = (123),\gamma_4= (123).\]
    These give all the four base cases and how we ``decompose'' any         $(\PP^1,\vec{p})$ with ramification data $\vec{\gamma}$.
    \end{proof}

    \begin{rmk}
        Again, we have \cite[remark 1.4]{deshpande.mukhopadhyay:2019} and factorization of coinvariants to establish that
        \[\VVd (\mfS_3; \Lambda_o,...,\Lambda_o)_{(\Sigma,\vec{q})\xrightarrow{\vec{\gamma}}(\PP^1,\vec{p})}\supseteq \bigotimes \VVd(\mfS_3;\Lambda_o,...,\Lambda_o)_{(\Sigma_i,\vec{q}_i)\xrightarrow{\vec{\gamma}_i}C_i}.\]
        Thus, we only need to compute the dimension for the $\mfS_3$-cover of each base cases.
    \end{rmk}

    \subsection{Proof of Theorem \ref{appthm:Iwa}}  We conclude the appendix with the proof of \Cref{appthm:Iwa}. Since the generic splitting degree is 6, it follows that $\Gamma = \mfS_3$. From the reduction discussed in the previous section, we are reduced to compute the dimension of  twisted conformal blocks in four base cases. Using \Cref{thm:criterion}, as long as none of them is zero, then the theorem holds.
    
    For case (1), due to the lack of 3-ramification data, we know all the $\mfS_3$-covering are going to be a 2-to-1 covering $\bar{\Sigma}$ followed by the trivial (disconnected) 3-to-1 \'etale covering of $\bar{\Sigma}$.
    The treatment in {\sectDes} showed that $\VVd\neq 0$ in this case.
    Similarly for case (2), due to the lack of 2-ramification data, we know all the $\mfS_3$-covering is a trivial \'etale 2-to-1 covering $\bar{\Sigma}$ followed by a 3-to-1 covering of $\PP^1$ repeated on both component of the trivial \'etale 2-to-1 covering.
    The treatment in {\sectDes} showed that $\VVd\neq 0$ in this case as well.
    
    Finally, we compute cases (3) and (4) using the Verlinde formula \cite[(1.2)]{deshpande.mukhopadhyay:2019}. 
    For level $\ell=1$, we have 
    \[\{0\}\subseteq P_1(D_4)^{\mfS_3} \subseteq P_1(D_4)^{(123)}=\{0\},\] 
    so the trivial representation is the only representation fixed by the diagram automorphism group $\mfS_3$ of $D_4$, and only one summand appear in the Verlinde formula. We only need to find each factor of the products namely: $S_{0,0}$, $S^{(12)}_{0,0}$, and $S^{(123)}_{0,0}$; the discussion at the end of \cite[Section 9]{deshpande.mukhopadhyay:2019} allows us to identify $S^{(12)}$ with $S^{(23)}$.
    From \cite[Example B.5]{deshpande.mukhopadhyay:2019}, we deduce that  $S_{0,0}=\frac{1}{2}$.
    Using \cite[Example B.10]{deshpande.mukhopadhyay:2019}, we have 
    $(S^{(123)}_{0,0})^3 = 2\cdot S_{(0,0)}=1$. It is clear that the trivial representation $0$ is the unit in the relevant fusion ring, so \cite[Remark 10.4]{deshpande.mukhopadhyay:2019} tells us that $S^{(123)}_{0,0}$ is a positive real number; combining these two facts together we have $S^{(123)}_{0,0} =1$.
    From \cite[Appendix $B$]{deshpande.mukhopadhyay:2019} we also get $P_1(D_4^{(2)})^{(12)} = \{0,\omega_3\}$. From section 4 of \textit{ibid.}, we get $\tilde{S}^{(12)}_{0,0}=1$ and $\tilde{S}^{(12)}_{0,\omega_3}\in \{\pm 1\}$ is a second root of unity, and thus the categorical $\dim \mathcal{C}_1 = 2$ (see \cite[Theorem 4.20]{deshpande.mukhopadhyay:2019} or \cite{Drinfeld.Gelaki.Nikshych.Ostrik:2010} for details).
    This gives us $S^{(12)}_{0,0} = \frac{1}{\sqrt{2}}$ by \cite[Definition 4.21]{deshpande.mukhopadhyay:2019}.
    We now substitute all the data into the Verlinde formula. In case (3) we have
    \[\rank[\VV_{\Sigma_3\ra \PP^1;(12),(23),(132)}(\vec{\Lambda}_0)] = \frac{S^{(12)}_{0,0}S^{(23)}_{0,0}S^{(132)}_{0,0}}{S_{0,0}} = \frac{\frac{1}{\sqrt{2}}\cdot \frac{1}{\sqrt{2}}\cdot 1}{\frac{1}{2}} = 1;\]
   and in case (4) we have
    \[\rank[\VV_{\Sigma_4\ra \PP^1;(12),(23),(123),(123)}(\vec{\Lambda}_0)] = \frac{S^{(12)}_{0,0}S^{(23)}_{0,0}S^{(123)}_{0,0}S^{(123)}_{0,0}}{(S_{0,0})^2} = \frac{\frac{1}{\sqrt{2}}\cdot \frac{1}{\sqrt{2}}\cdot 1\cdot 1}{(\frac{1}{2})^2} = 2.\]
   Combining this with the strategy of proof of  {\Iwanum}, we have proved \ref{appthm:Iwa}. \hfill$\square$

\bibliographystyle{alpha}
\bibliography{biblio}

\vfill
\end{document}